\documentclass[12pt]{amsart}
\usepackage{amsmath,amsfonts,amsthm,amsopn,amssymb,mathtools,stmaryrd}
\usepackage{cite,marginnote}
\usepackage{pdfsync}

\usepackage{color,enumitem,graphicx}
\usepackage[colorlinks=true,urlcolor=blue,
citecolor=red,linkcolor=blue,linktocpage,pdfpagelabels,
bookmarksnumbered,bookmarksopen]{hyperref}
\usepackage[english]{babel}

\usepackage[left=2.75cm,right=2.75cm,top=3cm,bottom=3cm]{geometry}
\usepackage[]{xcolor}

\def\sideremark#1{\ifvmode\leavevmode\fi\vadjust{\vbox to0pt{\vss
 \hbox to 0pt{\hskip\hsize\hskip1em
 \vbox{\hsize2.1cm\tiny\raggedright\pretolerance10000
  \noindent #1\hfill}\hss}\vbox to15pt{\vfil}\vss}}}%

\numberwithin{equation}{section}

\makeindex
\newtheorem{theorem}{Theorem}[section]
\newtheorem{proposition}[theorem]{Proposition}
\newtheorem{lemma}[theorem]{Lemma}
\newtheorem{remark}[theorem]{Remark}
\newtheorem{example}[theorem]{Example}
\newtheorem{corollary}[theorem]{Corollary}
\newtheorem{definition}[theorem]{Definition}

\newcommand{\ud}{\mathrm{d}}

\newcommand{\bt}{\begin{theorem}}
\newcommand{\et}{\end{theorem}}
\newcommand{\bl}{\begin{lemma}}
\newcommand{\el}{\end{lemma}}
\newcommand{\bd}{\begin{definition}}
\newcommand{\ed}{\end{definition}}
\newcommand{\bc}{\begin{corollary}}
\newcommand{\ec}{\end{corollary}}
\newcommand{\bp}{\begin{proof}}
\newcommand{\ep}{\end{proof}}
\newcommand{\bx}{\begin{example}}
\newcommand{\ex}{\end{example}}
\newcommand{\bi}{\begin{exercise}}
\newcommand{\ei}{\end{exercise}}
\newcommand{\bo}{\begin{proposition}}
\newcommand{\eo}{\end{proposition}}
\newcommand{\br}{\begin{remark}}
\newcommand{\er}{\end{remark}}
\newcommand{\be}{\begin{equation}}
\newcommand{\ee}{\end{equation}}
\newcommand{\ba}{\begin{align}}
\newcommand{\ea}{\end{align}}
\newcommand{\bn}{\begin{enumerate}}
\newcommand{\en}{\end{enumerate}}
\newcommand{\bg}{\begin{align*}}
\newcommand{\bcs}{\begin{cases}}
\newcommand{\ecs}{\end{cases}}

\newcommand{\bean}{\begin{eqnarray*}}
\newcommand{\eean}{\end{eqnarray*}}

\newtheorem{Lemma}{Lemma}[part]
\newtheorem{Proposition}{Proposition}[part]
\newtheorem{Remark}{Remark}[part]
\newtheorem{Theorem}{Theorem}[part]

\numberwithin{Assumption}{section} \numberwithin{Corollary}{section}
\numberwithin{Definition}{section} \numberwithin{equation}{section}
\numberwithin{Example}{section} \numberwithin{Lemma}{section}
\numberwithin{Proposition}{section} \numberwithin{Remark}{section}
\numberwithin{Theorem}{section}

\renewcommand{\leq}{\leqslant}
\renewcommand{\geq}{\geqslant}

\title[The existence and multiplicity of solutions]{The existence and multiplicity of solutions for general quasi-linear elliptic equations with sub-cubic nonlinearity}

\author[C.~Huang]{Chen Huang}
\author[J.~Zhang]{Jianjun Zhang}
\author[X.~X.~Zhong]{Xuexiu Zhong}

\address[C. Huang]{\newline\indent College of Science
\newline\indent
University of Shanghai for Science and Technology
\newline\indent Shanghai 200093, P.R. China}
\email{\href{mailto:chenhuangmath111@163.com}{chenhuangmath111@163.com}}

\address[J.~Zhang]{\newline\indent College of Mathematics and Statistics
\newline\indent
Chongqing Jiaotong University
\newline\indent
Chongqing 400074, PR China}
\email{\href{mailto:zhangjianjun09@tsinghua.org.cn}{zhangjianjun09@tsinghua.org.cn}}

\address[X.~X.~Zhong]{\newline\indent South China Research Center for Applied Mathematics and Interdisciplinary Studies
\newline\indent
South China Normal University
\newline\indent
Guangzhou 510631, PR China}
\email{\href{mailto:zhongxuexiu1989@163.com}{zhongxuexiu1989@163.com}}

\thanks{C. Huang is supported by Postdoctoral Science Foundation of China (2020M682065). J. Zhang is supported by NSFC(No.11871123). Xuexiu Zhong was supported by the NSFC (No.11801581), Guangdong Basic and Applied Basic Research Foundation (2021A1515010034),Guangzhou Basic and Applied Basic Research Foundation(No.202102020225).}

\date{\today}

\begin{document}

\begin{abstract}
We consider the existence and multiplicity of solutions for a class of quasi-linear Schr\"{o}dinger equations which include the modified nonlinear Schr\"{o}dinger equations. A new perturbation approach is used to treat the sub-cubic nonlinearity.

\end{abstract}
\maketitle

\section{ Introduction }
\setcounter{equation}{0}
\setcounter{Assumption}{0} \setcounter{Theorem}{0}
\setcounter{Proposition}{0} \setcounter{Corollary}{0}
\setcounter{Lemma}{0}\setcounter{Remark}{0}
\par

\indent\indent
In this paper, we focus on the existence and multiplicity of solutions for a class of quasi-linear elliptic equations of the form
\begin{equation}\label{eq1.1}
-\sum_{i,j=1}^{N}D_{j}(a_{ij}(u)D_{i}u)+\frac{1}{2}\sum_{i,j=1}^{N}D_{s}a_{ij}(u)D_{i}uD_{j}u+V(x)u=f(u),\ x\in\mathbb{R}^N,
\end{equation}
where $N\geq3$, $D_i=\frac{\partial}{\partial x_i}$, $D_{s}a_{ij}(s)=\frac{d}{ds}a_{ij}(s)$, $V(x)$ satisfies $(V_1)$ and $(V_2)$.\\
$(V_1)$ $V\in C(\mathbb{R}^N)\cap L^{\infty}(\mathbb{R}^N)$, $V(x)=V(|x|)$ and $\inf\limits_{x\in\mathbb{R}^N}V(x)>0$;\\
$(V_2)$ $V$ is weakly differentiable and
$$\left(\frac{1}{2}-\frac{1}{\mu}\right)V(x)-\frac{1}{2}(x\cdot\nabla V(x))\geq 0,\ a.e.\ x\in\mathbb{R}^N,$$
where $\mu$ is given in $(f_{3})$ as follows.\\
Moreover, we assume $f\in C(\mathbb{R},\mathbb{R})$ satisfies\\
$(f_1)$ $\lim\limits_{t\to0}\frac{f(t)}{t}=0$;\\
$(f_2)$ $\limsup\limits_{|t|\to\infty}\frac{|f(t)|}{|t|^{q-1}}<\infty$ for some $q\in(2,\frac{4N}{N-2})$;\\
$(f_3)$ there exists $4>\mu>2$ such that
$$tf(t)\geq\mu F(t)>0\ \mbox{for}\ t\neq0,\ \mbox{where}\ F(t)=\int_{0}^{t}f(s)ds;$$
$(f_4)$ $f$ is odd.\\
For the functions $a_{ij}(i,j=1,\cdots,N)$, we assume the following conditions hold.\\
$(a_1)$ $a_{ij}\in C^{1}(\mathbb{R},\mathbb{R})$, $a_{ij}=a_{ji}$ and $D_{s}a_{ij}$ satisfies the uniformly Lipschitz condition, that is, there exists a constant $C_1$ such that
$$|D_{s}a_{ij}(s_1)-D_{s}a_{ij}(s_2)|\leq C_{1}|s_1-s_2|,\ \mbox{for all}\ s_1,s_2\in\mathbb{R}.$$
$(a_2)$ There exists a constant $C_2>0$ such that
$$C_{2}^{-1}(1+s^2)|\xi|^{2}\leq\sum\limits_{i,j=1}^{N}a_{ij}(s)\xi_{i}\xi_{j}\leq C_{2}(1+s^2)|\xi|^{2},\ \mbox{for all}\ s\in\mathbb{R},\ \xi\in\mathbb{R}^N.$$
$(a_3)$ There exists a positive constant $1\leq C_3<\min\{\frac{3\mu}{2},\frac{\beta}{2}\}$ with $\beta\in(\max\{4,q\},\frac{4N}{N-2})$ such that for all $s\in\mathbb{R}$, $\xi\in\mathbb{R}^N$,
$$\sum_{i,j=1}^{N}a_{ij}(s)\xi_{i}\xi_{j}\leq\sum_{i,j=1}^{N}\left[a_{ij}(s)+\frac{1}{2}D_{s}a_{ij}(s)s\right]\xi_{i}\xi_{j}\leq C_{3}\sum_{i,j=1}^{N}a_{ij}(s)\xi_{i}\xi_{j}.$$
$(a_4)$ $a_{ij}$ is even.

A typical example of $a_{ij}$ is $a_{ij}(s)=(1+2s^{2})\delta_{ij}$, which corresponds with the so-called the modified nonlinear Schr\"{o}dinger equations
\begin{equation}\label{eq1.2}
-\Delta u+V(x)u-\frac{1}{2}\Delta(u^{2})u=f(u),\ x\in\mathbb{R}^N.
\end{equation}
Note that problem (\ref{eq1.2}) appears
in many models from mathematical physics (\cite{Bor1,Bra1,Hasse1,Kur1,Laedke1,Lange1,Makhankov1}).

A major difficulty associated with problem (\ref{eq1.2}) is that the associated natural functional $Q:H^{1}(\mathbb{R}^N)\to\mathbb{R}$ given by
$$Q(u)=\frac{1}{2}\int_{\mathbb{R}^N}\left(|\nabla u|^{2}+V(x)|u|^{2}\right)dx+\int_{\mathbb{R}^N}|\nabla u|^{2}|u|^{2}dx-\int_{\mathbb{R}^N}F(u)dx,$$
is not well defined on $H^{1}(\mathbb{R}^N)$ for $N\geq3$. In fact, due to the presence of the quasi-linear term, there is no natural function space in which $Q$ is both well defined and possesses compactness properties. For this reason, problem (\ref{eq1.2}) has received considerable attention in mathematical analysis in the last decades. To the best of our knowledge, the early existence result for equations of form (\ref{eq1.2}) is due to Poppenberg $et\ al.$ \cite{Pop1}, where the authors proved the existence of standing wave solutions for the quasi-linear Schr\"odinger equation containing strongly singular nonlinearities. Next, Liu $et\ al.$ \cite{Liu1} introduced a utilizing variable substitution (the dual approach) and convert quasi-linear problem (\ref{eq1.2}) into a semi-linear one in an Orlicz space framework. Colin $et\ al.$ in \cite{Col1} adopted a similar method of variable substitution which does not use Orlicz spaces, but the classical Sobolev space $H^{1}(\mathbb{R}^N)$. For further results, we refer to \cite{Ada1,Alv2,Huang1,Huang2,Liu2,LiuX3,LiuZ2,Shen1,Sil1,Wan1}. Furthermore, we remark that in most of the aforementioned references,
the power-like nonlinearity $f(u)=|u|^{p-2}u$ is involved and the restriction $p>4$ plays a crucial role in ensuring the boundedness of (PS) sequences to the energy functional. Recently, Jing $et\ al.$ \cite{Jing2} and Zhang $et\ al.$ \cite{Zhang1} proposed a new variational approach to deal with $p\in(2,\frac{4N}{N-2})$ in a unified way and obtained infinitely many sign-changing solutions.

As the above references, the dual approach has been extensively applied in the modified nonlinear Schr\"{o}dinger equations. But the dual approach seems difficult to be used in the general quasi-linear elliptic problem (\ref{eq1.1}). For this reason, limited work has been done in the general form of the quasi-linear problem (\ref{eq1.1}). In \cite{Liu2}, a least energy sign-changing solution of (\ref{eq1.1}) is obtained via the Nehari manifold method.  In \cite{LiuJ1}, the existence of multi-bump solutions was shown for the quasi-linear elliptic problem (\ref{eq1.1}). Particularly, a truncation argument is introduced so that the problem can be dealt with in $H^{1}(\mathbb{R}^N)$. Since this truncated functional is not of class $C^{1}$, a non-smooth critical point theory was employed. Multiple solutions of (\ref{eq1.1}) were first obtained in \cite{LiuX1}, where a 4-Laplacian perturbation term is involved in (\ref{eq1.1}), so that the associated functionals are well-defined in $W^{1,4}_{0}(\Omega)$. Then the classical critical point theory can be applied to obtain the existence of an infinite sequence of solutions to the perturbed equations, and solutions of the original problem are obtained by passing to the limit. Then in \cite{LiuX2}, the authors obtained multiplicity of sign-changing solutions for a general form of the quasi-linear problem (\ref{eq1.1}). This idea is further developed in \cite{LiuX3}, where the critical case was treated and some new existence results were given. Most of previous results mentioned above yield only with the power-like nonlinearity $f(u)=|u|^{q-2}u$ $(4<q\leq 2\cdot2^{\ast})$. The case $1<q<2$ was investigated in \cite{LiuZ2}, where by using the variants of Clark's theorem, problem (\ref{eq1.1}) has a sequence of solutions with $L^{\infty}$-norms tending to zero. But for the case $2<q<4$, less results are known. By using the perturbation approach and the invariant sets approach, in \cite{Jing1}  Jing $et\ al.$  showed the structure of solutions depends keenly upon the parameter $\lambda$. For $2<q<4$ and $\lambda$ large enough, they show that the following equation
$$
-\sum_{i,j=1}^{N}D_{j}(a_{ij}(x,u)D_{i}u)+\frac{1}{2}\sum_{i,j=1}^{N}D_{s}a_{ij}(x,u)D_{i}uD_{j}u=\lambda|u|^{q-2}u,\ x\in\Omega
$$
has multiple solutions in $H^{1}_{0}(\Omega)$. However no further information was shown on the sign of solutions obtained in \cite{Jing1}.  In \cite{LiuJ2}, as $\lambda$ becomes larger and larger, Liu $et\ al.$ proved the existence of more and more sign-changing solutions with positive or negative energies.

Motivated by the works described above, we are interested in quasi-linear elliptic problems of the form (\ref{eq1.1}) involving sub-cubic nonlinearities without the parameter $\lambda$. Our paper here provides a different convergence theorem for constructing solutions if $f(u)$ satisfies the sub-cubic growth.

Now, we outline our idea and approach. Formally, equation (\ref{eq1.1}) is called the Euler-Lagrange equation of the functional
\begin{equation}\label{eq1.3}
I(u)=\frac{1}{2}\int_{\mathbb{R}^N}\sum\limits_{i,j=1}^{N}a_{ij}(u)D_{i}uD_{j}u\ud x+\frac{1}{2}\int_{\mathbb{R}^N}V(x)|u|^{2}\ud x-\int_{\mathbb{R}^N}F(u)\ud x.
\end{equation}
A function $u\in H^{1}(\mathbb{R}^N)\cap L^{\infty}(\mathbb{R}^N)$ is called a critical point of $I$ if $u$ satisfies $\langle I'(u),\varphi\rangle=0$ for all $\varphi\in C^{\infty}_{0}(\mathbb{R}^N)$, that is, $u$ is also a weak solution of (\ref{eq1.1}). Notice that the functional $I$ is not differentiable in $H^{1}(\mathbb{R}^N)$. In \cite{LiuX1,LiuX2,LiuX3} authors introduced a $p$-Laplacian perturbation method to deal with the quasi-linear problem (\ref{eq1.1}) involving a super-cubic nonlinearity. In \cite{LiuX1} they considered the functional
\begin{equation*}
\begin{aligned}
\widetilde{I}_{\lambda}(u)&=\frac{\lambda}{4}\int_{\Omega}|Du|^{4}\ud x+\frac{1}{2}\int_{\Omega}\sum\limits_{i,j=1}^{N}a_{ij}(u)D_{i}uD_{j}u\ud x-\frac{1}{q}\int_{\Omega}|u|^{q}\ud x,\ u\in W^{1,4}_{0}(\Omega),
\end{aligned}
\end{equation*}
here $f(t)=|t|^{q-2}t$ ($4<q<\frac{4N}{N-2}$). In order to obtain sign-changing solutions, in \cite{LiuX2} they considered the functional
\begin{equation*}
\begin{aligned}
\widetilde{I}_{\lambda}(u)&=\frac{\lambda}{p}\int_{\Omega}|Du|^{p}\ud x+\frac{\lambda}{p}\int_{\Omega}|Du|^{p-2}|u|^{2}\ud x+\frac{1}{2}\int_{\Omega}\sum\limits_{i,j=1}^{N}a_{ij}(u)D_{i}uD_{j}u\ud x\\
&\indent-\frac{1}{q}\int_{\Omega}|u|^{q}\ud x,\ u\in W^{1,p}_{0}(\Omega),
\end{aligned}
\end{equation*}
where $4<p<q<\frac{4N}{N-2}$. In both cases the critical points of the perturbed functionals $\widetilde{I}_{\lambda}$ are used as approximate solutions of the original problem. For the case $f(t)=|t|^{q-2}t$ $(2<q<4)$, the situation is quite different. Since the perturbed functional does not enjoy a linking structure. As a result, the minimax argument can not be used directly. Following the idea of \cite{LiuZ3,LiuZ4,Zhang1}, we consider a family of perturbed functionals, for $0<\lambda\leq1$ and $u\in E$,
\begin{equation}\label{eq1.4}
\begin{aligned}
I_{\lambda}(u)&=\frac{\lambda}{4}\int_{\mathbb{R}^N}\left(|Du|^{4}+|u|^{4}\right)\ud x+\frac{1}{2}\int_{\mathbb{R}^N}\sum\limits_{i,j=1}^{N}a_{ij}(u)D_{i}uD_{j}u\ud x+\frac{1}{2}\int_{\mathbb{R}^N}V(x)|u|^{2}\ud x\\
&\indent-\int_{\mathbb{R}^N}F(u)\ud x+\frac{\lambda}{2(1+\alpha)}\left(\int_{\mathbb{R}^N}|u|^{2}\ud x\right)^{1+\alpha}-\frac{\lambda}{r}\int_{\mathbb{R}^N}|u|^{r}\ud x,
\end{aligned}
\end{equation}
where $\max\{4,q\}<\beta<r<\frac{4N}{N-2}$ and $0<\alpha<\frac{\mu-2}{\mu N+2}$.
Here $E:=W^{1,4}(\mathbb{R}^N)\cap H_{r}^{1}(\mathbb{R}^N)$,
$$W^{1,4}(\mathbb{R}^N)=\left\{u\in L^4(\mathbb{R}^N):\int_{\mathbb{R}^N}|Du|^{4}\ud x<+\infty\right\}$$
with the norm
$$\|u\|_{4}=\left(\int_{\mathbb{R}^N}(|Du|^{4}+u^{4})\ud x\right)^{\frac{1}{4}},$$
and
$$H^{1}_{r}(\mathbb{R}^N)=\{u\in H^{1}(\mathbb{R}^N):u(x)=u(|x|)\},$$
which is a Hilbert space endowed with the norm
$$\|u\|_{2}=\left(\int_{\mathbb{R}^N}(|Du|^{2}+u^{2})\ud x\right)^{\frac{1}{2}}.$$
Denote $\|u\|$ as the norm in $E$, then
$$\|u\|=\|u\|_{2}+\|u\|_{4}.$$
It is easy to know that $I_{\lambda}$ is a $C^{1}$ functional defined on $E$, and
\begin{equation*}
\begin{aligned}
\langle I_{\lambda}'(u),\varphi\rangle&=\lambda\int_{\mathbb{R}^N}\left(|Du|^{2}DuD\varphi+u^{3}\varphi\right)\ud x+\int_{\mathbb{R}^N}\sum\limits_{i,j=1}^{N}a_{ij}(u)D_{i}uD_{j}\varphi \ud x\\
&\indent+\frac{1}{2}\int_{\mathbb{R}^N}\sum\limits_{i,j=1}^{N}D_{s}a_{ij}(u)D_{i}uD_{j}u\varphi \ud x+\int_{\mathbb{R}^N}V(x)u\varphi \ud x\\
&\indent-\int_{\mathbb{R}^N}f(u)\varphi \ud x+\lambda\left(\int_{\mathbb{R}^N}u^{2}\ud x\right)^{\alpha}\int_{\mathbb{R}^N}u\varphi \ud x-\lambda\int_{\mathbb{R}^N}|u|^{r-2}u\varphi \ud x,
\end{aligned}
\end{equation*}
for all $\varphi\in E$.
\begin{Remark}\label{rem1.1}
Motivated by \cite{LiuX1} and \cite{LiuZ4}, thanks to the three perturbation terms, we can overcome the following three main difficulties.\\
$(a)$ Adding $4$-Laplacian perturbation to overcome the non-differentiability of the functional $I$ in $H^{1}(\mathbb{R}^N)$.\\
$(b)$ Using a higher order term $\frac{\lambda}{r}\int_{\mathbb{R}^N}|u|^{r}\ud x$ with $r>4$ to recover the linking structure.\\
$(c)$ Since the case $2<\mu<4$, without the classical Ambrosetti-Rabinowitz condition for the quasi-linear elliptic problem (\ref{eq1.1}), it seems tough to obtain the boundedness of Palais-Smale sequences. To overcome this obstacle, we give a higher order term $\frac{\lambda}{r}\int_{\mathbb{R}^N}|u|^{r}\ud x$ with $r>4$ and another perturbation term $\frac{\lambda}{2(1+\alpha)}\left(\int_{\mathbb{R}^N}u^{2}\ud x\right)^{1+\alpha}$ in the associated natural functional.
\end{Remark}

In the next section, we establish appropriate estimates, that is to say, the critical points of $I_{\lambda}$ converge to ones of $I$ as $\lambda\to0^{+}$. Now, we are ready to state our main results.
\begin{Theorem}\label{thm1.1}
Under the assumptions $(V_1)$, $(V_2)$, $(a_1)$-$(a_3)$ and $(f_1)$-$(f_3)$, problem (\ref{eq1.1}) has a positive solution $u\in H_{r}^{1}(\mathbb{R}^{N})\cap L^{\infty}(\mathbb{R}^N)$.
\end{Theorem}
\begin{Theorem}\label{thm1.2}
Suppose that conditions $(V_1)$, $(V_2)$, $(a_1)$-$(a_4)$ and $(f_1)$-$(f_4)$ hold, then for any $\lambda\in(0,1],$ $I_{\lambda}$ has a sequence of critical points $\{u_{\lambda,j}\}_{j=1}^{\infty}$. Moreover, as $\lambda\to0^{+}$, $u_{\lambda,j}$ converges to a solution $u_{0,j}$ of problem (\ref{eq1.1}) with $I(u_{0,j})\to\infty$ as $j\to\infty$.
\end{Theorem}

This paper is organized as follows. In Section 2, we establish a convergence result with a sub-cubic nonlinearity. The proofs of the existence of the mountain pass type critical points of (\ref{eq1.1}) is obtained in Sections 3. Finally, Theorem \ref{thm1.2} is proved in Section 4.\\
\indent In what follows, we use $\|\cdot\|$ as the norm in $E$, $\|\cdot\|_{2}$ as the norm in $H^{1}_{r}(\mathbb{R}^N)$, $\|\cdot\|_{4}$ as the norm in $W^{1,4}(\mathbb{R}^N)$ and $|\cdot|_{p}$ as the norm in $L^{p}(\mathbb{R}^N)$ for $2\leq p\leq+\infty$. Define $u^{\pm}=\max\{\pm u,0\}$. $C$ and $C_{i}\ (i=1,2,\cdots)$ denote positive generic constants.

\section{The Perturbed Functionals and Convergence} \setcounter{equation}{0}
\setcounter{Assumption}{0}
\setcounter{Theorem}{0}
\setcounter{Proposition}{0}
\setcounter{Corollary}{0}
\setcounter{Lemma}{0}
\indent\indent In this section, we prove a convergence result of the perturbed functionals which will be used later in obtaining solutions of (\ref{eq1.1}) by passing to the limit as $\lambda\to0^{+}$. To prove the convergence result, we need the following lemmas.
\begin{Lemma}\label{lem2.0}
Let $u$ be a critical point of $I_{\lambda}$ in $E$ for $\lambda\in(0,1]$, then
\begin{equation*}
\begin{aligned}
&\frac{\lambda(4-N)}{4}\int_{\mathbb{R}^{N}}|D u|^{4}\ud x-\frac{\lambda N}{4}\int_{\mathbb{R}^{N}}u^{4}\ud x+\frac{2-N}{2}\int_{\mathbb{R}^N}\sum\limits_{i,j=1}^{N} a_{ij}(u)D_{i}uD_{j}u\ud x\\
&\indent-\frac{1}{2}\int_{\mathbb{R}^N}\left(NV(x)+(x\cdot \nabla V(x))^{+}\right)u^{2}\ud x
+N\int_{\mathbb{R}^N}F(u)\ud x-\frac{N\lambda}{2}\left(\int_{\mathbb{R}^N}u^{2}\ud x\right)^{1+\alpha}\\
&\indent+\frac{N\lambda}{r}\int_{\mathbb{R}^N}|u|^{r}\ud x\leq0.
\end{aligned}
\end{equation*}
\end{Lemma}
\begin{proof}
Since $u$ is critical point of $I_{\lambda}$, by the classical elliptic regularity theorem (see \cite{Gil1}), we get $u\in W^{2,4}_{loc}(\mathbb{R}^N)\cap E$. And then we have the following Pohozaev type inequality, whose proof is standard and can be referred
to Lemma 2.6 in \cite{Xu1} for example.
\begin{equation*}
\begin{aligned}
0&\geq\frac{\lambda(4-N)}{4}\int_{\mathbb{R}^{N}}|D u|^{4}\ud x-\frac{\lambda N}{4}\int_{\mathbb{R}^{N}}u^{4}\ud x+\frac{2-N}{2}\int_{\mathbb{R}^N}\sum\limits_{i,j=1}^{N} a_{ij}(u)D_{i}uD_{j}u\ud x\\
&\indent-\frac{1}{2}\int_{\mathbb{R}^N}\left(NV(x)+(x\cdot\nabla V(x))^{+}\right)u^{2}\ud x
+N\int_{\mathbb{R}^N}F(u)\ud x-\frac{N\lambda}{2}\left(\int_{\mathbb{R}^N}u^{2}\ud x\right)^{1+\alpha}\\
&\indent+\frac{N\lambda}{r}\int_{\mathbb{R}^N}|u|^{r}\ud x.
\end{aligned}
\end{equation*}
In fact, due to condition $(V_2)$, we know that
$$(x\cdot\nabla V(x))^{+}\in L^{\infty}(\mathbb{R}^N).$$
Hence $\int_{\mathbb{R}^N}(x\cdot \nabla V(x))^{+}u^{2}\ud x$ is well defined.
\end{proof}
\begin{Lemma}\label{lem2.1}
Given $M>0$ there exists $C=C(M)>0$ (independent of $\lambda>0$) such that if $u$ is a critical point of $I_{\lambda}$ with $I_{\lambda}(u)\leq M$, then $\lambda\int_{\mathbb{R}^N}\left(|Du|^{4}+u^{4}\right)\ud x\leq C$, $\int_{\mathbb{R}^N}(1+u^{2})|Du|^{2}\ud x\leq C$ and $\int_{\mathbb{R}^N}u^{2}\ud x\leq C$.
\end{Lemma}
\begin{proof}
By the definition of $I_{\lambda}$, we obtain
\begin{equation}\label{eq2.1}
\begin{aligned}
M&\geq\frac{\lambda}{4}\int_{\mathbb{R}^N}\left(|Du|^{4}+u^{4}\right)\ud x+\frac{1}{2}\int_{\mathbb{R}^N}\sum\limits_{i,j=1}^{N}a_{ij}(u)D_{i}uD_{j}u\ud x+\frac{1}{2}\int_{\mathbb{R}^N}V(x)u^{2}\ud x\\
&\indent-\int_{\mathbb{R}^N}F(u)\ud x+\frac{\lambda}{2(1+\alpha)}\left(\int_{\mathbb{R}^N}u^{2}\ud x\right)^{1+\alpha}-\frac{\lambda}{r}\int_{\mathbb{R}^N}|u|^{r}\ud x
\end{aligned}
\end{equation}
and
\begin{equation}\label{eq2.2}
\begin{aligned}
0&=\langle I_{\lambda}'(u),u\rangle=\lambda\int_{\mathbb{R}^N}\left(|Du|^{4}+u^{4}\right)\ud x+\int_{\mathbb{R}^N}\sum\limits_{i,j=1}^{N}a_{ij}(u)D_{i}uD_{j}u \ud x\\
&\indent+\frac{1}{2}\int_{\mathbb{R}^N}\sum\limits_{i,j=1}^{N}D_{s}a_{ij}(u)uD_{i}uD_{j}u \ud x+\int_{\mathbb{R}^N}V(x)u^{2} \ud x\\
&\indent-\int_{\mathbb{R}^N}f(u)u\ud x+\lambda\left(\int_{\mathbb{R}^N}u^{2}\ud x\right)^{1+\alpha}-\lambda\int_{\mathbb{R}^N}|u|^{r}\ud x.
\end{aligned}
\end{equation}
Moreover, by Lemma \ref{lem2.0}, we have the following Pohozaev type identity.
\begin{equation}\label{eq2.3}
\begin{aligned}
0\geq&\frac{\lambda(4-N)}{4}\int_{\mathbb{R}^{N}}|D u|^{4}\ud x-\frac{\lambda N}{4}\int_{\mathbb{R}^{N}}u^{4}\ud x+\frac{2-N}{2}\int_{\mathbb{R}^N}\sum\limits_{i,j=1}^{N} a_{ij}(u)D_{i}uD_{j}u\ud x\\
&-\frac{1}{2}\int_{\mathbb{R}^N}\left(NV(x)+(x\cdot\nabla V(x))^{+}\right)u^{2}\ud x
+N\int_{\mathbb{R}^N}F(u)\ud x\\
&-\frac{N\lambda}{2}\left(\int_{\mathbb{R}^N}u^{2}\ud x\right)^{1+\alpha}+\frac{N\lambda}{r}\int_{\mathbb{R}^N}|u|^{r}\ud x.
\end{aligned}
\end{equation}
Multiplying (\ref{eq2.1}), (\ref{eq2.2}) and (\ref{eq2.3}) by $N+1$, $-\frac{1}{\mu}$ and $1$ respectively and adding them up, we get
\begin{equation}\label{eq2.4}
\begin{aligned}
(N+1)M\geq&\lambda\left(\frac{5}{4}-\frac{1}{\mu}\right)\int_{\mathbb{R}^N}|D u|^{4}\ud x+\lambda\left(\frac{1}{4}-\frac{1}{\mu}\right)\int_{\mathbb{R}^N}u^{4}\ud x\\
&+\lambda\left(\frac{\mu-\alpha \mu N-2(1+\alpha)}{2(1+\alpha)\mu}\right)\left(\int_{\mathbb{R}^N}u^{2}\ud x\right)^{1+\alpha}+\lambda\left(\frac{1}{\mu}-\frac{1}{r}\right)\int_{\mathbb{R}^N}|u|^{r}\ud x
\\
&+\left(\frac{1}{2}-\frac{1}{\mu}\right)\int_{\mathbb{R}^N} V(x)u^{2}\ud x
-\frac{1}{2}\int_{\mathbb{R}^N}(x\cdot\nabla V(x))^{+}u^{2}\ud x\\
&+\left(\frac{3}{2}-\frac{1}{\mu}\right)\int_{\mathbb{R}^N}\sum\limits_{i,j=1}^{N} a_{ij}(u)D_{i}uD_{j}u\ud x
\\
&-\frac{1}{2\mu}\int_{\mathbb{R}^N}\sum\limits_{i,j=1}^{N} D_{s}a_{ij}(u)uD_{i}uD_{j}u\ud x.
\end{aligned}
\end{equation}
By conditions $(V_2)$ and $(a_3)$, we get
\begin{equation*}
\begin{aligned}
(N+1)M\geq&\lambda\left(\frac{5}{4}-\frac{1}{\mu}\right)\int_{\mathbb{R}^N}|D u|^{4}\ud x+\lambda\left(\frac{1}{4}-\frac{1}{\mu}\right)\int_{\mathbb{R}^N}u^{4}\ud x\\
&\indent+\lambda\left(\frac{\mu-\alpha\mu N-2(1+\alpha)}{2(1+\alpha)\mu}\right)\left(\int_{\mathbb{R}^N}u^{2}\ud x\right)^{1+\alpha}\\
&\indent+\lambda\left(\frac{1}{\mu}-\frac{1}{r}\right)\int_{\mathbb{R}^N}|u|^{r}\ud x.
\end{aligned}
\end{equation*}
Note that, for any $A_1>0$, we can choose $A_2>0$ such that
$$t^{1+\alpha}>A_{1}t-A_2,\ \mbox{for all}\ t\geq0.$$
Applying this with $t=|u|_{2}^{2}$, then we have
\begin{equation*}
\begin{aligned}
&\indent\lambda\left(\frac{1}{4}-\frac{1}{\mu}\right)\int_{\mathbb{R}^N}u^{4}\ud x+\lambda\left(\frac{\mu-\alpha \mu N-2(1+\alpha)}{2(1+\alpha)\mu}\right)\left(\int_{\mathbb{R}^N}u^{2}\ud x\right)^{1+\alpha}\\
&\indent+\lambda\left(\frac{1}{\mu}-\frac{1}{r}\right)\int_{\mathbb{R}^N}|u|^{r}\ud x\\
&\geq\lambda\left(\frac{1}{4}-\frac{1}{\mu}\right)\int_{\mathbb{R}^N}u^{4}\ud x+\lambda\left(\frac{\mu-\alpha \mu N-2(1+\alpha)}{2(1+\alpha)\mu}\right)\left(A_{1}\int_{\mathbb{R}^N}u^{2}\ud x-A_2\right)\\
&\indent+\lambda\left(\frac{1}{\mu}-\frac{1}{r}\right)\int_{\mathbb{R}^N}|u|^{r}\ud x.
\end{aligned}
\end{equation*}
Since $2<\mu<4<r$, $\frac{1}{\mu}-\frac{1}{r}>0$. From $0<\alpha<\frac{\mu-2}{\mu N+2}$, it implies $\frac{\mu-\alpha \mu N-2(1+\alpha)}{2(1+\alpha)\mu}>0$. Then we take $A_1$ large enough such that the function
\begin{equation*}
\begin{aligned}
\left(\frac{\mu-\alpha \mu N-2(1+\alpha)}{2(1+\alpha)\mu}\right)A_{1}t^{2}+\left(\frac{1}{\mu}-\frac{1}{r}\right)t^{r}-\left(\frac{1}{\mu}-\frac{1}{4}\right)t^{4}\ge0,\,t\ge0.
\end{aligned}
\end{equation*}
Therefore
\begin{equation*}
\begin{aligned}
&\left(\frac{\mu-\alpha \mu N-2(1+\alpha)}{2(1+\alpha)\mu}\right)\left(\int_{\mathbb{R}^N}u^{2}\ud x\right)^{1+\alpha}+\left(\frac{1}{\mu}-\frac{1}{r}\right)\int_{\mathbb{R}^N}|u|^{r}\ud x\\
&\indent-\left(\frac{1}{\mu}-\frac{1}{4}\right)\int_{\mathbb{R}^N}u^{4}\ud x\geq-\left(\frac{\mu-\alpha\mu N-2(1+\alpha)}{2(1+\alpha)\mu}\right)A_{2}.
\end{aligned}
\end{equation*}
Hence, $\lambda\int_{\mathbb{R}^N}|D u|^{4}\ud x\leq C(M)$. From (\ref{eq2.4}), $(V_2)$ and the above estimates, we obtain
\begin{equation*}
\begin{aligned}
(N+1)M\geq&-\lambda\left(\frac{\mu-\alpha\mu N-2(1+\alpha)}{2(1+\alpha)\mu}\right) A_{2}+C_0\int_{\mathbb{R}^N}u^{2}\ud x\\
&\indent+\left(\frac{3}{2}-\frac{1}{\mu}\right)\int_{\mathbb{R}^N}\sum\limits_{i,j=1}^{N} a_{ij}(u)D_{i}uD_{j}u\ud x
\\
&\indent-\frac{1}{2\mu}\int_{\mathbb{R}^N}\sum\limits_{i,j=1}^{N} D_{s}a_{ij}(u)uD_{i}uD_{j}u\ud x.
\end{aligned}
\end{equation*}
From this and condition $(a_3)$, we have $\int_{\mathbb{R}^N}(1+u^{2})|Du|^{2}\ud x\leq C(M)$ and $\int_{\mathbb{R}^N} u^{2}\ud x\leq C(M)$. The interpolation inequality is applied to show $|u|_{p}\leq C(M)$($2\leq p\leq\frac{4N}{N-2}$). Especially, we also have $\int_{\mathbb{R}^N}u^{4}\ud x\leq C(M)$.
\end{proof}
The proof of Theorem \ref{thm1.1} and \ref{thm1.2} is based on the following convergence result for the perturbed variational problem $I_{\lambda}$. Before giving the proof, we show the following lower semi-continuity result.

\begin{Lemma}\label{lem2.3}
Let the symmetric matrix $A(t)\in \mathbb{R}^{N\times N}$ be non-negatively definite and continuous in $t\in\mathbb{R}$. Assume that there exist $u_{0}\in H_r^{1}(\mathbb{R}^N)$ and $\{u_n\}\subset H_r^{1}(\mathbb{R}^N)$ such that $u_n(x)\in[-C,C]$ for some $C>0$(independent of $n$), $u_n\rightharpoonup u_{0}$ in $H_r^{1}(\mathbb{R}^N)$, $u_n(x)\to u_{0}(x)$ a.e. $x\in\mathbb{R}^N$, then
$$
\liminf_{n\rightarrow\infty}\int_{\mathbb{R}^N}(A(u_n)\nabla u_n\cdot \nabla u_n)\ud x\geq\int_{\mathbb{R}^N}(A(u_0)\nabla u_0\cdot \nabla u_0)\ud x.
$$
\end{Lemma}
\begin{proof}
Let $v_n=u_n-u_0$, then
\begin{align*}
&\int_{\mathbb{R}^N}(A(u_n)\nabla u_n\cdot \nabla u_n)\ud x=\int_{\mathbb{R}^N}(A(u_n)\nabla u_0\cdot \nabla u_0)\ud x\\
&+\int_{\mathbb{R}^N}(A(u_n)\nabla v_n\cdot \nabla v_n)\ud x+2\int_{\mathbb{R}^N}(A(u_n)\nabla u_0\cdot \nabla v_n)\ud x\\
&\ge\int_{\mathbb{R}^N}(A(u_n)\nabla u_0\cdot \nabla u_0)\ud x+2\int_{\mathbb{R}^N}(A(u_n)\nabla v_n\cdot \nabla u_0)\ud x.
\end{align*}
By Fatou's Lemma,
$$
\liminf_{n\rightarrow\infty}\int_{\mathbb{R}^N}(A(u_n)\nabla u_0\cdot \nabla u_0)\ud x\ge\int_{\mathbb{R}^N}(A(u_0)\nabla u_0\cdot \nabla u_0)\ud x.
$$
Now, we only need to show
$$
\lim_{n\rightarrow\infty}\int_{\mathbb{R}^N}(A(u_n)\nabla v_n\cdot \nabla u_0)\ud x=0.
$$
Due to continuity of $A$, for any $\varepsilon>0$, there exists $\delta>0$ such that, for any $\overrightarrow{a},\overrightarrow{b}\in\mathbb{R}^N$, $$\left|([A(t_1)-A(t_2)]\overrightarrow{a}\cdot\overrightarrow{b})\right|\leq \varepsilon|\overrightarrow{a}||\overrightarrow{b}|,\,\,\mbox{if}\,\,t_1,t_2\in[-C,C],\,|t_1-t_2|\leq \delta.$$
Set $E_n(\delta):=\{x\in\mathbb{R}^N\}:|u_n(x)-u_0(x)|\leq\delta$, then up to a subsequence, $meas(\mathbb{R}^N\setminus E_n(\delta))\rightarrow0$ as $n\rightarrow\infty$. It follows that
$$
\lim_{n\rightarrow\infty}\int_{\mathbb{R}^N\setminus E_n(\delta)}(A(u_n)\nabla v_n\cdot \nabla u_0)\ud x=0.
$$
On the one hand, there exists $c>0$(independent of $n,\varepsilon$) such that
\begin{align*}
&\left|\int_{E_n(\delta)}(A(u_n)\nabla v_n\cdot \nabla u_0)\ud x\right|\\
&\leq\left|\int_{E_n(\delta)}([A(u_n)-A(u_0)]\nabla v_n\cdot \nabla u_0)\ud x\right|+\left|\int_{E_n(\delta)}(A(u_0)\nabla v_n\cdot \nabla u_0)\ud x\right|\\
&\leq\varepsilon\int_{\mathbb{R}^N}|\nabla v_n||\nabla u_0|\ud x+\left|\int_{E_n(\delta)}(A(u_0)\nabla v_n\cdot \nabla u_0)\ud x\right|\\
&\leq\left|\int_{E_n(\delta)}(A(u_0)\nabla v_n\cdot \nabla u_0)\ud x\right|+c\varepsilon.
\end{align*}
On the other hand, since $meas(\mathbb{R}^N\setminus E_n(\delta))\rightarrow0$ and $v_n\rightharpoonup 0$ in $H^1(\mathbb{R}^N)$ as $n\rightarrow\infty$, we have
\begin{align*}
&\left|\int_{E_n(\delta)}(A(u_0)\nabla v_n\cdot \nabla u_0)\ud x\right|\\
&\leq\left|\int_{\mathbb{R}^N}(A(u_0)\nabla v_n\cdot \nabla u_0)\ud x\right|+\left|\int_{\mathbb{R}^N\setminus E_n(\delta)}(A(u_0)\nabla v_n\cdot \nabla u_0)\ud x\right|\\
&=o_n(1).
\end{align*}
Then we have
$$
\limsup_{n\rightarrow\infty}\left|\int_{\mathbb{R}^N}(A(u_n)\nabla v_n\cdot \nabla u_0)\ud x\right|\leq c\varepsilon.
$$
Since $\varepsilon$ is arbitrary, the desired result is concluded.
\end{proof}

Motivated
by a similar convergence result in \cite{LiuX2}, we have the following result.
\begin{Proposition}\label{pro2.1}
Suppose that $\lambda_n\to0^{+}$$(n\to\infty)$, $u_n$ is a critical point of $I_{\lambda_n}$ and $I_{\lambda_n}(u_n)\leq C$ for some $C>0$ independent of $n$, then there exist a subsequence of $\{u_n\}$ and a critical point $u_{0}\in H_{r}^{1}(\mathbb{R}^N)\cap L^{\infty}(\mathbb{R}^N)$ of $I$ such that $u_n\to u_{0}$ in $H_{r}^{1}(\mathbb{R}^N)$, $u_{n}Du_{n}\to u_{0}Du_{0}$ in $L^{2}(\mathbb{R}^N)$ and $\lambda_{n}\int_{\mathbb{R}^N}(|Du_n|^{4}+|u_n|^{4})\ud x\to0$, and $I_{\lambda_n}(u_n)\to I(u_{0})$.
\end{Proposition}
\begin{proof}
By Lemma \ref{lem2.1}, there exist $u_{0}\in H_{r}^{1}(\mathbb{R}^N)$ and $\{u_n\}\subset H_{r}^{1}(\mathbb{R}^N)$ such that $u_n\rightharpoonup u_{0}$ in $H_{r}^{1}(\mathbb{R}^N)$, $u_{n}Du_{n}\rightharpoonup u_{0}Du_{0}$ in $L^{2}(\mathbb{R}^N)$, $u_n\to u_{0}$ in $L^{p}(\mathbb{R}^N)$ with $p\in(2,\frac{4N}{N-2})$, $u_n(x)\to u_{0}(x)$ a.e. $x\in\mathbb{R}^N$ and $|u_n|_{4N/(N-2)}\leq C$.\\
{\bf Step 1.} $\{u_n\}$ is bounded in $L^{\infty}(\mathbb{R}^N)$.\\
First, we claim that there exists $p>4N/(N-2)$ such that $|u_n|_{p}\leq C$. In fact, we have
\begin{equation*}
\begin{aligned}
&\indent\lambda_{n}\int_{\mathbb{R}^N}\left(|Du_n|^{2}Du_{n}D\varphi+u_{n}^{3}\varphi\right)\ud x+\int_{\mathbb{R}^N}\sum\limits_{i,j=1}^{N}a_{ij}(u_{n})D_{i}u_{n}D_{j}\varphi \ud x\\
&+\frac{1}{2}\int_{\mathbb{R}^N}\sum\limits_{i,j=1}^{N}D_{s}a_{ij}(u_{n})D_{i}u_{n}D_{j}u_{n}\varphi \ud x+\int_{\mathbb{R}^N}V(x)u_{n}\varphi \ud x+\lambda_{n}\left(\int_{\mathbb{R}^N}u_{n}^{2}\ud x\right)^{\alpha}\int_{\mathbb{R}^N}u_{n}\varphi \ud x\\
&=\int_{\mathbb{R}^N}f(u_{n})\varphi \ud x+\lambda_{n}\int_{\mathbb{R}^N}|u_{n}|^{r-2}u_{n}\varphi \ud x,\ \mbox{for all}\ \varphi\in E.
\end{aligned}
\end{equation*}
Taking $\varphi=|u_{n}^{T}|^{2(\eta-1)}u_n$ with $\eta>1$, where $u_{n}^{T}=u_n$, if $|u_n|\leq T$; $u_{n}^{T}=T$, if $u_n\geq T$; $u_{n}^{T}=-T$, if $u_n\leq -T$, then we have
\begin{equation}\label{eq2.5}
\begin{aligned}
&\indent\int_{|u_n|\leq T}\sum\limits_{i,j=1}^{N}a_{ij}(u_{n})D_{i}u_{n}D_{j}(|u_{n}^{T}|^{2(\eta-1)})u_n \ud x+\int_{\mathbb{R}^N}\sum\limits_{i,j=1}^{N}a_{ij}(u_{n})D_{i}u_{n}D_{j}u_n |u_{n}^{T}|^{2(\eta-1)}\ud x\\
&+\frac{1}{2}\int_{\mathbb{R}^N}\sum\limits_{i,j=1}^{N}D_{s}a_{ij}(u_{n})D_{i}u_{n}D_{j}u_{n}|u_{n}^{T}|^{2(\eta-1)}u_n \ud x+\int_{\mathbb{R}^N}V(x)u^{2}_{n}|u_{n}^{T}|^{2(\eta-1)}\ud x\\
&\leq\int_{\mathbb{R}^N}f(u_{n})u_{n}|u_{n}^{T}|^{2(\eta-1)} \ud x+\int_{\mathbb{R}^N}|u_{n}|^{r}|u_{n}^{T}|^{2(\eta-1)} \ud x.
\end{aligned}
\end{equation}
The left hand of the above inequality has the following estimates
\begin{equation}\label{eq2.6}
\begin{aligned}
&\indent\int_{|u_n|\leq T}\sum\limits_{i,j=1}^{N}a_{ij}(u_{n})D_{i}u_{n}D_{j}(|u_{n}^{T}|^{2(\eta-1)})u_n \ud x+\int_{\mathbb{R}^N}\sum\limits_{i,j=1}^{N}a_{ij}(u_{n})D_{i}u_{n}D_{j}u_n |u_{n}^{T}|^{2(\eta-1)}\ud x\\
&\indent\indent+\frac{1}{2}\int_{\mathbb{R}^N}\sum\limits_{i,j=1}^{N}D_{s}a_{ij}(u_{n})D_{i}u_{n}D_{j}u_{n}|u_{n}^{T}|^{2(\eta-1)}u_n \ud x+\int_{\mathbb{R}^N}V(x)u^{2}_{n}|u_{n}^{T}|^{2(\eta-1)}\ud x\\
&\geq\int_{|u_n|\leq T}\sum\limits_{i,j=1}^{N}\left((2\eta-1)a_{ij}(u_{n})+\frac{1}{2}D_{s}a_{ij}(u_{n})u_n\right)D_{i}u_{n}D_{j}u_n |u_{n}^{T}|^{2(\eta-1)} \ud x\\
&\indent\indent+\frac{1}{2}\int_{|u_n|\geq T}\sum\limits_{i,j=1}^{N}D_{s}a_{ij}(u_{n})u_n D_{i}u_{n}D_{j}u_{n}T^{2(\eta-1)} \ud x+\int_{\mathbb{R}^N}V(x)u^{2}_{n}|u_{n}^{T}|^{2(\eta-1)}\ud x\\
&\geq C\int_{|u_n|\leq T}(1+2u_{n}^{2})|Du_n|^{2}|u_n^{T}|^{2(\eta-1)}\ud x+\int_{\mathbb{R}^N}V(x)u^{2}_{n}|u_{n}^{T}|^{2(\eta-1)}\ud x\\
&\geq C\int_{|u_n|\leq T}u_{n}^{2}|Du_n|^{2}|u_n^{T}|^{2(\eta-1)}\ud x+\int_{\mathbb{R}^N}V(x)u^{2}_{n}|u_{n}^{T}|^{2(\eta-1)}\ud x.
\end{aligned}
\end{equation}
On the other hand, choose $0<\delta<\frac{1}{2}V_0$. There exists $C_{\delta}>0$ such that
\begin{equation}\label{eq2.7}
\begin{aligned}
\int_{\mathbb{R}^N}f(u_{n})u_{n}|u_{n}^{T}|^{2(\eta-1)} \ud x\leq\delta\int_{\mathbb{R}^N}u^{2}_{n}|u_{n}^{T}|^{2(\eta-1)}\ud x+C_{\delta}\int_{\mathbb{R}^N}|u_{n}|^{r}|u_{n}^{T}|^{2(\eta-1)} \ud x.
\end{aligned}
\end{equation}
By (\ref{eq2.5})-(\ref{eq2.7}), we have
\begin{equation}\label{eq2.8}
\begin{aligned}
C\int_{|u_n|\leq T}u_{n}^{2}|Du_n|^{2}|u_n^{T}|^{2(\eta-1)}\ud x&\leq\int_{\mathbb{R}^N}|u_{n}|^{r}|u_{n}^{T}|^{2(\eta-1)} \ud x\\
&\leq\int_{\mathbb{R}^N}|u_{n}|^{r}|u_{n}|^{2(\eta-1)} \ud x\\
&\leq|u_n|_{\frac{4N}{N-2}}^{r}|u_n|_{\frac{8N(\eta-1)}{4N-r(N-2)}}^{2(\eta-1)}.
\end{aligned}
\end{equation}
Furthermore, by Sobolev's embedding theorems,
\begin{equation}\label{eq2.9}
\begin{aligned}
\int_{|u_n|\leq T}u_{n}^{2}|Du_n|^{2}|u_n^{T}|^{2(\eta-1)}\ud x&=\left(\frac{1}{\eta+1}\right)^{2}\int_{\mathbb{R}^N}|D(|u_{n}^T|^{\eta+1})|^{2}\ud x\\
&\geq C\left(\frac{1}{\eta+1}\right)^{2}\left(\int_{|u_n|\leq T}|u_{n}|^{(\eta+1)\cdot\frac{2N}{N-2}}\ud x\right)^{\frac{N-2}{N}}.
\end{aligned}
\end{equation}
From (\ref{eq2.8}) and (\ref{eq2.9}), we have
$$\left(\int_{|u_n|\leq T}|u_{n}|^{(\eta+1)\cdot\frac{2N}{N-2}}\ud x\right)^{\frac{N-2}{N}}\leq C(\eta+1)^{2}|u_n|_{\frac{8N(\eta-1)}{4N-r(N-2)}}^{2(\eta-1)}.$$
Letting $T\to+\infty$, it follows that
$$|u_n|_{(\eta+1)\cdot\frac{2N}{N-2}}^{2(\eta+1)}\leq C(\eta+1)^{2}|u_n|_{\frac{8N(\eta-1)}{4N-r(N-2)}}^{2(\eta-1)}.$$
Set $\eta_0=\frac{2\cdot2^{\ast}-r}{2}+1>1$, then
$$|u_n|_{(\eta_0+1)\cdot\frac{2N}{N-2}}^{2(\eta_0+1)}\leq C(\eta_{0}+1)^{2}|u_n|_{\frac{4N}{N-2}}^{2(\eta_{0}-1)},$$
thus $$p=(\eta_{0}+1)\cdot\frac{2N}{N-2}>\frac{4N}{N-2}$$ is what we need. Next by using the Moser's iteration, starting from $|u_n|_{4N/(N-2)}<C$, we know that $u_{n}\in L^{\infty}(\mathbb{R}^N)$ and $|u_{n}|_{\infty}\leq C$, where $C$ depends on $\sup_n|u_{n}|_{4N/(N-2)}$ only. These can be proved in the standard way.
\\
{\bf Step 2.} $u_{0}\in H_{r}^{1}(\mathbb{R}^N)\cap L^{\infty}({\mathbb{R}^N})$ is a critical point of $I$.\\
Take $\varphi\in C^{\infty}_{0}(\mathbb{R}^N)$, $\varphi\geq0$ and $\psi=\varphi\exp(-Hu_{n})$ where $H>0$ is large enough such that $-Ha_{ij}(u_{n})+\frac{1}{2}D_{s}a_{ij}(u_{n})$ is negatively definite. From $|u_{n}|_{\infty}\leq C$, we have $\psi\in H_{r}^{1}(\mathbb{R}^N)\cap L^{\infty}(\mathbb{R}^N)$. Since $u_{n}$ is a critical point of (\ref{eq1.4}), then
{\allowdisplaybreaks
\begin{equation*}
\begin{aligned}
0&=\lambda_n\int_{\mathbb{R}^N}\left(|Du_n|^{2}Du_nD\psi+u_n^{3}\psi\right)\ud x+\int_{\mathbb{R}^N}V(x)u_n\psi \ud x\\
&\indent+\int_{\mathbb{R}^N}\sum\limits_{i,j=1}^{N}a_{ij}(u_{n})D_{i}u_{n}D_{j}\psi \ud x+\frac{1}{2}\int_{\mathbb{R}^N}\sum\limits_{i,j=1}^{N}D_{s}a_{ij}(u_{n})D_{i}u_{n}D_{j}u_{n}\psi \ud x\\
&\indent-\int_{\mathbb{R}^N}f(u_n)\psi \ud x+\lambda_n\left(\int_{\mathbb{R}^N}u_{n}^{2}\ud x\right)^{\alpha}\int_{\mathbb{R}^N}u_{n}\psi \ud x-\lambda_n\int_{\mathbb{R}^N}|u_{n}|^{r-2}u_{n}\psi \ud x
\end{aligned}
\end{equation*}
That is,
\begin{equation*}
\begin{aligned}
&0=\lambda_n\int_{\mathbb{R}^N}\left(|Du_n|^{2}Du_nD(\varphi\exp(-Hu_{n}))+u_n^{3}\varphi\exp(-Hu_{n})\right)\ud x\\
&\indent+\int_{\mathbb{R}^N}\sum\limits_{i,j=1}^{N}a_{ij}(u_{n})D_{i}u_{n}D_{j}\varphi\exp(-Hu_{n})\ud x\\
&\indent+\int_{\mathbb{R}^N}\sum\limits_{i,j=1}^{N}\left(-Ha_{ij}(u_{n})+\frac{1}{2}D_{s}a_{ij}(u_{n})\right)D_{i}u_{n}D_{j}u_{n}\varphi\exp(-Hu_{n})\ud x\\
&\indent+\int_{\mathbb{R}^N}V(x)u_n\varphi\exp(-Hu_{n}) \ud x-\int_{\mathbb{R}^N}f(u_n)\varphi\exp(-Hu_{n}) \ud x\\
&\indent+\lambda_n\left(\int_{\mathbb{R}^N}u_{n}^{2}\ud x\right)^{\alpha}\int_{\mathbb{R}^N}u_{n}\varphi\exp(-Hu_{n})\ud x
-\lambda_n\int_{\mathbb{R}^N}|u_{n}|^{r-2}u_{n}\varphi\exp(-Hu_{n})\ud x
\end{aligned}
\end{equation*}
So
\begin{equation*}
\begin{aligned}
&0\leq\int_{\mathbb{R}^N}\sum\limits_{i,j=1}^{N}a_{ij}(u_{0})D_{i}u_{0}D_{j}\varphi\exp(-Hu_{0})\ud x\\
&\indent+\int_{\mathbb{R}^N}\sum\limits_{i,j=1}^{N}\left(-Ha_{ij}(u_{0})+\frac{1}{2}D_{s}a_{ij}(u_{0})\right)D_{i}u_{0}D_{j}u_{0}\varphi\exp(-Hu_{0})\ud x\\
&\indent+\int_{\mathbb{R}^N}V(x)u_0\varphi\exp(-Hu_{0})\ud x-\int_{\mathbb{R}^N}f(u_0)\varphi\exp(-Hu_{0})\ud x+o_{n}(1)\\
&=\int_{\mathbb{R}^N}\sum\limits_{i,j=1}^{N}a_{ij}(u_{0})D_{i}u_{0}D_{j}(\varphi\exp(-Hu_0)) \ud x\\
&\indent+\frac{1}{2}\int_{\mathbb{R}^N}\sum\limits_{i,j=1}^{N}D_{s}a_{ij}(u_{0})D_{i}u_{0}D_{j}u_{0}\varphi\exp(-Hu_0)\ud x\\
&\indent+\int_{\mathbb{R}^N}V(x)u_0\varphi\exp(-Hu_{0})\ud x-\int_{\mathbb{R}^N}f(u_0)\varphi\exp(-Hu_{0})\ud x+o_{n}(1),
\end{aligned}
\end{equation*}
}%
where we used the Fatou's Lemma and Lemma \ref{lem2.3}. Indeed, we also used
\begin{equation*}
\begin{aligned}
&\indent\lambda_{n}\left|\int_{\mathbb{R}^N}|D u_n|^{2}Du_nD\varphi \exp(-Hu_n)\ud x\right|\\
&\leq\lambda_{n}\left(\int_{\mathbb{R}^N}|Du_n|^{4}\ud x\right)^{\frac{3}{4}}\left(\int_{\mathbb{R}^N}|D\varphi|^{4}\exp(-Hu_n)^{4}\ud x\right)^{\frac{1}{4}}\to0,\ \mbox{as}\ n\to+\infty;
\end{aligned}
\end{equation*}
\begin{equation*}
\begin{aligned}
&\indent\lambda_{n}\left|\int_{\mathbb{R}^N}u_n^{3}\varphi\exp(-Hu_n)\ud x\right|\\
&\leq\lambda_{n}\left(\int_{\mathbb{R}^N}|u_n|^{4}\ud x\right)^{\frac{3}{4}}\left(\int_{\mathbb{R}^N}|\varphi|^{4}\exp(-Hu_n)^{4}\ud x\right)^{\frac{1}{4}}\to0,\ \mbox{as}\ n\to+\infty;
\end{aligned}
\end{equation*}
$$-\lambda_nH\int_{\mathbb{R}^N}|Du_n|^{4}\varphi\exp(-Hu_{n})\ud x\leq0;$$
$$\int_{\mathbb{R}^N}V(x)u_n\varphi\exp(-Hu_{n}) \ud x\to\int_{\mathbb{R}^N}V(x)u_0\varphi\exp(-Hu_{0})\ud x,\ \mbox{as}\ n\to+\infty;$$
$$\int_{\mathbb{R}^N}f(u_n)\varphi\exp(-Hu_{n}) \ud x\to\int_{\mathbb{R}^N}f(u_0)\varphi\exp(-Hu_{0})\ud x,\ \mbox{as}\ n\to+\infty;$$
$$\lambda_n\left(\int_{\mathbb{R}^N}u_{n}^{2}\ud x\right)^{\alpha}\int_{\mathbb{R}^N}u_{n}\varphi\exp(-Hu_{n})\ud x\to0,\ \mbox{as}\ n\to+\infty;$$
$$\lambda_n\int_{\mathbb{R}^N}|u_{n}|^{r-2}u_{n}\varphi\exp(-Hu_{n})\ud x\to 0,\ \mbox{as}\ n\to+\infty.$$
Thus for all $\varphi\in C^{\infty}_{0}(\mathbb{R}^N)$ and $\varphi\geq0$, we have
\begin{equation}\label{eq2.10}
\begin{aligned}
0&\leq\int_{\mathbb{R}^N}\sum\limits_{i,j=1}^{N}a_{ij}(u_{0})D_{i}u_{0}D_{j}(\varphi\exp(-Hu_0)) \ud x\\
&\indent+\frac{1}{2}\int_{\mathbb{R}^N}\sum\limits_{i,j=1}^{N}D_{s}a_{ij}(u_{0})D_{i}u_{0}D_{j}u_{0}\varphi\exp(-Hu_0)\ud x\\
&\indent+\int_{\mathbb{R}^N}V(x)u_0\varphi\exp(-Hu_{0})\ud x-\int_{\mathbb{R}^N}f(u_0)\varphi\exp(-Hu_{0})\ud x.
\end{aligned}
\end{equation}
For all $0\leq\xi\in C^{\infty}_{0}(\mathbb{R}^N)$, $u_{0}\in L^{\infty}(\mathbb{R}^N)$ implies that $\xi\exp(Hu_0)\in H^{1}(\mathbb{R}^N)$. We take a sequence $0\leq\varphi_n\in C^{\infty}_{0}(\mathbb{R}^N)$ such that
\begin{equation*}
\begin{aligned}
\varphi_n&\to\xi\exp(Hu_0)\ \mbox{in}\ H^{1}(\mathbb{R}^N)\, \mbox{and a.e.}\ x\in\mathbb{R}^N.
\end{aligned}
\end{equation*}
We can choose $\varphi=\varphi_n$ in (\ref{eq2.10}). As $n\to\infty$, we have
\begin{equation*}
\begin{aligned}
0&\leq\int_{\mathbb{R}^N}\sum\limits_{i,j=1}^{N}a_{ij}(u_{0})D_{i}u_{0}D_{j}\xi \ud x+\frac{1}{2}\int_{\mathbb{R}^N}\sum\limits_{i,j=1}^{N}D_{s}a_{ij}(u_{0})D_{i}u_{0}D_{j}u_{0}\xi \ud x\\
&\indent+\int_{\mathbb{R}^N}V(x)u_0\xi \ud x-\int_{\mathbb{R}^N}f(u_0)\xi \ud x,\ \mbox{for all}\ \xi\in C^{\infty}_{0}(\mathbb{R}^N)\ \mbox{and}\ \xi\geq0.
\end{aligned}
\end{equation*}
Similarly, by choosing $\psi=\varphi\exp(Hu_{n})$, we can get an opposite inequality. Hence, for all $\xi\in C^{\infty}_{0}(\mathbb{R}^N)$, we have
\begin{equation}\label{eq2.11}
\begin{aligned}
&\int_{\mathbb{R}^N}\sum\limits_{i,j=1}^{N}a_{ij}(u_{0})D_{i}u_{0}D_{j}\xi \ud x+\frac{1}{2}\int_{\mathbb{R}^N}\sum\limits_{i,j=1}^{N}D_{s}a_{ij}(u_{0})D_{i}u_{0}D_{j}u_{0}\xi \ud x\\
&\indent+\int_{\mathbb{R}^N}V(x)u_0\xi \ud x-\int_{\mathbb{R}^N}f(u_0)\xi \ud x=0.
\end{aligned}
\end{equation}
{\bf Step 3.} Convergence results.\\
Due to approximations, we can take $\xi=u_0$ in (\ref{eq2.11}) as follows.
\begin{equation}\label{eq2.12}
\begin{aligned}
&\int_{\mathbb{R}^N}\sum\limits_{i,j=1}^{N}a_{ij}(u_{0})D_{i}u_{0}D_{j}u_0 \ud x+\frac{1}{2}\int_{\mathbb{R}^N}\sum\limits_{i,j=1}^{N}D_{s}a_{ij}(u_{0})u_0 D_{i}u_{0}D_{j}u_{0} \ud x\\
&\indent+\int_{\mathbb{R}^N}V(x)u_0^{2} \ud x-\int_{\mathbb{R}^N}f(u_0)u_0\ud x=0.
\end{aligned}
\end{equation}
Recall that $u_n$ is a critical point of $I_{\lambda_n}$, we obtain
\begin{equation}\label{eq2.13}
\begin{aligned}
0&=\lambda_{n}\int_{\mathbb{R}^N}\left(|Du_{n}|^{4}+u_{n}^{4}\right)\ud x+\int_{\mathbb{R}^N}\sum\limits_{i,j=1}^{N}a_{ij}(u_{n})D_{i}u_{n}D_{j}u_{n} \ud x\\
&\indent+\frac{1}{2}\int_{\mathbb{R}^N}\sum\limits_{i,j=1}^{N}D_{s}a_{ij}(u_{n})u_{n} D_{i}u_{n}D_{j}u_{n} \ud x+\int_{\mathbb{R}^N}V(x)u_{n}^{2} \ud x\\
&\indent-\int_{\mathbb{R}^N}f(u_n)u_n\ud x+\lambda_{n}\left(\int_{\mathbb{R}^N}u_{n}^{2}\ud x\right)^{1+\alpha} -\lambda_{n}\int_{\mathbb{R}^N}|u_{n}|^{r}\ud x.
\end{aligned}
\end{equation}
From $u_n\to u_{0}$ in $L^{p}(\mathbb{R}^N)$ with $p\in(2,\frac{4N}{N-2})$, $\int_{\mathbb{R}^N}V(x)u_n^{2}\ud x\leq C$ and $\int_{\mathbb{R}^N}(1+u_{n}^{2})|Du_{n}|^{2}\ud x\leq C$, we have
$$
\int_{\mathbb{R}^N}f(u_n)u_n \ud x\to\int_{\mathbb{R}^N}f(u_0)u_0\ud x,\,\,
\lambda_n\left(\int_{\mathbb{R}^N}u_{n}^{2}\ud x\right)^{1+\alpha}\to0,\,\,
\lambda_n\int_{\mathbb{R}^N}|u_{n}|^{r}\ud x\to0.
$$
By all the above estimates and Lemma \ref{lem2.3}, we have
$$\lambda_{n}\int_{\mathbb{R}^N}\left(|Du_{n}|^{4}+u_{n}^{4}\right)\ud x\to0,\ \mbox{as}\ n\to\infty$$
and
\begin{equation*}
\begin{aligned}
&\int_{\mathbb{R}^N}\sum\limits_{i,j=1}^{N}\left(a_{ij}(u_{n})+\frac{1}{2}D_{s}a_{ij}(u_{n})u_{n}\right)D_{i}u_{n}D_{j}u_{n}\ud x+\int_{\mathbb{R}^N}V(x)u_{n}^{2} \ud x\\
&\indent\to\int_{\mathbb{R}^N}\sum\limits_{i,j=1}^{N}\left(a_{ij}(u_{0})+\frac{1}{2}D_{s}a_{ij}(u_{0})u_0\right)D_{i}u_{0}D_{j}u_{0}\ud x+\int_{\mathbb{R}^N}V(x)u_{0}^{2}\ud x,\ \mbox{as}\ n\to\infty.
\end{aligned}
\end{equation*}
Thus $I_{\lambda_{n}}(u_{n})\to I(u_{0})$, as $n\to\infty$.
\end{proof}

\section{The Mountain Pass Type Critical Point for Perturbed Functionals} \setcounter{equation}{0}
\setcounter{Assumption}{0}
\setcounter{Theorem}{0}
\setcounter{Proposition}{0}
\setcounter{Corollary}{0}
\setcounter{Lemma}{0}
\indent\indent We prove a compactness condition for $I_{\lambda}$ which will be used later.
\begin{Lemma}\label{lem3.1}
For any $\lambda\in(0,1]$ fixed, $I_{\lambda}$ satisfies the Palais-Smale condition.
\end{Lemma}
\begin{proof}
Assume that $\{u_n\}\subset E$ is a (PS) sequence of $I_{\lambda}$, that is, $|I_{\lambda}(u_n)|\leq C$, $I_{\lambda}'(u_n)\to0$ as $n\to\infty$. Recalling that $\beta\in(\max\{4,q\},r)$, we have
\begin{equation}\label{eq3.1}
\begin{aligned}
C+o_n(1)\|u_n\|&\geq I_{\lambda}(u_{n})-\frac{1}{\beta}\langle I'_{\lambda}(u_{n}),u_{n}\rangle\\
&=\left(\frac{1}{4}-\frac{1}{\beta}\right)\lambda\int_{\mathbb{R}^N}(|Du_{n}|^{4}+u_{n}^{4})\ud x\\
&\indent+\int_{\mathbb{R}^N}\sum\limits_{i,j=1}^{N}\left(\left(\frac{1}{2}-\frac{1}{\beta}\right)a_{ij}(u_{n})-\frac{1}{2\beta}D_{s}a_{ij}(u_{n})u_{n}\right)D_{i}u_{n}D_{j}u_{n}\ud x\\
&\indent+\left(\frac{1}{2}-\frac{1}{\beta}\right)\int_{\mathbb{R}^N}V(x)u_{n}^{2}\ud x+\int_{\mathbb{R}^N}\left(\frac{1}{\beta}f(u_n)u_n-F(u_n)\right)\ud x\\
&\indent+\left(\frac{1}{2(1+\alpha)}-\frac{1}{\beta}\right)\lambda\left(\int_{\mathbb{R}^N}u_{n}^{2}\ud x\right)^{1+\alpha}+\left(\frac{1}{\beta}-\frac{1}{r}\right)\lambda\int_{\mathbb{R}^N}|u_{n}|^{r}\ud x\\
&\geq \left(\frac{1}{4}-\frac{1}{\beta}\right)\lambda\int_{\mathbb{R}^N}(|Du_{n}|^{4}+u_{n}^{4})\ud x\\
&\indent+C\int_{\mathbb{R}^N}(1+u_{n}^{2})|Du_{n}|^{2}\ud x+\left(\frac{1}{2}-\frac{1}{\beta}-\delta\right)\int_{\mathbb{R}^N}V(x)u_{n}^{2}\ud x\\ &\indent-C_{\delta}\int_{\mathbb{R}^N}|u_{n}|^{q}\ud x+\left(\frac{1}{2(1+\alpha)}-\frac{1}{\beta}\right)\lambda\left(\int_{\mathbb{R}^N}u_{n}^{2}\ud x\right)^{1+\alpha}\\
&\indent+\left(\frac{1}{\beta}-\frac{1}{r}\right)\lambda\int_{\mathbb{R}^N}|u_{n}|^{r}\ud x,
\end{aligned}
\end{equation}
where $\delta>0$ is small enough.

Now we claim that $\{u_n\}$ is bounded in $E$. In fact, for any $A_1>0$, we can choose $A_2>0$ such that
$$t^{1+\alpha}>A_{1}t-A_2,\ \mbox{for}\ t>0.$$
Apply this with $t=|u_n|_{2}^{2}$ and $\delta=\frac{1}{4}-\frac{1}{2\beta}$, then by (\ref{eq3.1}) we have
\begin{equation}\label{eq3.2}
\begin{aligned}
C+o_n(1)\|u_n\|&\geq \left(\frac{1}{4}-\frac{1}{\beta}\right)\lambda\int_{\mathbb{R}^N}(|Du_{n}|^{4}+u_{n}^{4})\ud x\\
&\indent+C\int_{\mathbb{R}^N}(1+u_{n}^{2})|Du_{n}|^{2}\ud x+\left(\frac{1}{4}-\frac{1}{2\beta}\right)\int_{\mathbb{R}^N}V(x)u_{n}^{2}\ud x\\ &\indent+\int_{\mathbb{R}^N}\left[\left(\frac{1}{\beta}-\frac{1}{r}\right)\lambda|u_{n}|^{r}+\left(\frac{1}{2(1+\alpha)}-\frac{1}{\beta}\right)\lambda A_{1}u^{2}_{n}-C|u_{n}|^{q}\right]\ud x\\
&\indent-\left(\frac{1}{2(1+\alpha)}-\frac{1}{\beta}\right)\lambda A_2.
\end{aligned}
\end{equation}
Since $2<\mu<4\leq\max\{4,q\}<\beta<r$ and $0<\alpha<\frac{\mu-2}{\mu N+2}$, then $\frac{1}{\beta}-\frac{1}{r}>0$, $\frac{1}{2(1+\alpha)}-\frac{1}{\beta}>0$ and $\frac{1}{q}-\frac{1}{\beta}>0$. We take $A_1$ large enough such that the function $$\left(\frac{1}{\beta}-\frac{1}{r}\right)\lambda t^{r}+\left(\frac{1}{2(1+\alpha)}-\frac{1}{\beta}\right)\lambda A_{1}t^{2}-Ct^{q}>0,\ \mbox{for any}\ t>0.$$
Hence $\{u_n\}$ is bounded in $E$. Up to a subsequence, we may assume that $u_n\rightharpoonup u$ in $E$. Since $E$ is radical, the embedding of $E\hookrightarrow L^{p}(\mathbb{R}^N)$ is compact ($2< p<\frac{4N}{N-2}$). Then
{\allowdisplaybreaks
\begin{equation}\label{eq3.3}
\begin{aligned}
&\indent\langle I'_{\lambda}(u_{n})-I'_{\lambda}(u),u_{n}-u\rangle\\
&=\lambda\int_{\mathbb{R}^N}\left((|Du_{n}|^{2}Du_n-|Du|^{2}Du)\cdot D(u_n-u)+(u_{n}^{3}-u^{3})(u_n-u)\right)\ud x\\
&\indent+\int_{\mathbb{R}^N}\sum\limits_{i,j=1}^{N}\left(a_{ij}(u_{n})D_{i}u_n-a_{ij}(u)D_{i}u\right)(D_{j}u_{n}-D_{j}u)\ud x\\
&\indent+\frac{1}{2}\int_{\mathbb{R}^N}\sum\limits_{i,j=1}^{N}\left(D_{s}a_{ij}(u_{n})D_{i}u_nD_{j}u_{n}-D_{s}a_{ij}(u)D_{i}uD_{j}u\right)(u_n-u)\ud x\\
&\indent+\int_{\mathbb{R}^N}V(x)(u_{n}-u)^{2}\ud x-\int_{\mathbb{R}^N}(f(u_n)-f(u))(u_n-u)\ud x\\
&\indent+\lambda|u_n|_{2}^{2\alpha}\int_{\mathbb{R}^N}(u_n-u)^{2}\ud x+\lambda(|u_n|_{2}^{2\alpha}-|u|_{2}^{2\alpha})\int_{\mathbb{R}^N}u(u_n-u)\ud x\\
&\indent-\lambda\int_{\mathbb{R}^N}(|u_{n}|^{r-2}u_n-|u|^{r-2}u)(u_n-u)\ud x.
\end{aligned}
\end{equation}
}%
Now we estimate every term in (\ref{eq3.3}). By $(a_1)$ and $(a_3)$, we have
$$
\liminf_{n\rightarrow\infty}\int_{\mathbb{R}^N}\sum\limits_{i,j=1}^{N}\left(D_{s}a_{ij}(u_{n})D_{i}u_nD_{j}u_{n}-D_{s}a_{ij}(u)D_{i}uD_{j}u\right)(u_n-u)\ud x\geq 0.
$$
By the condition $(a_1)$ we obtain for some $\theta\in(0,1)$,
{\allowdisplaybreaks
$$
\aligned
&\indent\left|a_{ij}(u_n)-a_{ij}(u)\right|\\
&=|D_{s}a_{ij}(\theta u_n+(1-\theta)u)||u_n-u|\\
&\leq(|D_{s}a_{ij}(0)|+|D_{s}a_{ij}(\theta u_n+(1-\theta)u)-D_{s}a_{ij}(0)|)|u_n-u|\\
&\leq C(1+|u_n|+|u|)|u_n-u|,
\endaligned
$$
}%
and so
\begin{equation*}
\begin{aligned}
&\indent\int_{\mathbb{R}^N}\sum\limits_{i,j=1}^{N}\left(a_{ij}(u_{n})D_{i}u_n-a_{ij}(u)D_{i}u\right)(D_{j}u_{n}-D_{j}u)\ud x\\
&=\int_{\mathbb{R}^N}\sum\limits_{i,j=1}^{N}a_{ij}(u_n)(D_{i}u_n-D_{i}u)(D_{j}u_n-D_{j}u)\ud x\\
&\indent+\int_{\mathbb{R}^N}\sum\limits_{i,j=1}^{N}(a_{ij}(u_n)-a_{ij}(u))D_{i}u(D_{j}u_n-D_{j}u)\ud x\\
&\geq\int_{\mathbb{R}^N}\sum\limits_{i,j=1}^{N}a_{ij}(u_n)(D_{i}u_n-D_{i}u)(D_{j}u_n-D_{j}u)\ud x\\
&\,\,\,\,\,\,\,\,\,\,\,\,\,\,\,\,\,\,-C\int_{\mathbb{R}^N}(1+|u_n|+|u|)|u_n-u||Du||Du_n-Du|\ud x.
\end{aligned}
\end{equation*}
Then
\begin{equation*}
\begin{aligned}
&\indent\int_{\mathbb{R}^N}\sum\limits_{i,j=1}^{N}\left(a_{ij}(u_{n})D_{i}u_n-a_{ij}(u)D_{i}u\right)(D_{j}u_{n}-D_{j}u)\ud x\\
&\geq C\int_{\mathbb{R}^N}(1+u_{n}^{2})|D(u_n-u)|^{2}\ud x+o_{n}(1).
\end{aligned}
\end{equation*}
Since $u_n\rightharpoonup u$ in $L^{2}(\mathbb{R}^N)$, $|u_n|_{2}$ is bounded, one easily has
$$(|u_n|_{2}^{2\alpha}-|u|_{2}^{2\alpha})\int_{\mathbb{R}^N}u(u_n-u)\ud x=o_{n}(1).$$
Returning to (\ref{eq3.3}) we have
\begin{equation*}
\begin{aligned}
o_{n}(1)&=\lambda\int_{\mathbb{R}^N}\left((|Du_{n}|^{2}Du_n-|Du|^{2}Du)\cdot D(u_n-u)+(u_{n}^{3}-u^{3})(u_n-u)\right)\ud x\\
&\indent+\int_{\mathbb{R}^N}\sum\limits_{i,j=1}^{N}\left(a_{ij}(u_{n})D_{i}u_n-a_{ij}(u)D_{i}u\right)(D_{j}u_{n}-D_{j}u)\ud x\\
&\indent+\frac{1}{2}\int_{\mathbb{R}^N}\sum\limits_{i,j=1}^{N}\left(D_{s}a_{ij}(u_{n})D_{i}u_nD_{j}u_{n}-D_{s}a_{ij}(u)D_{i}uD_{j}u\right)(u_n-u)\ud x\\
&\indent+\int_{\mathbb{R}^N}V(x)(u_{n}-u)^{2}\ud x-\int_{\mathbb{R}^N}(f(u_n)-f(u))(u_n-u)\ud x\\
&\indent+\lambda|u_n|_{2}^{2\alpha}\int_{\mathbb{R}^N}(u_n-u)^{2}\ud x+\lambda(|u_n|_{2}^{2\alpha}-|u|_{2}^{2\alpha})\int_{\mathbb{R}^N}u(u_n-u)\ud x\\
&\indent-\lambda\int_{\mathbb{R}^N}(|u_{n}|^{r-2}u_n-|u|^{r-2}u)(u_n-u)\ud x,
\end{aligned}
\end{equation*}
which yields that
\begin{equation*}
\begin{aligned}
o_{n}(1)&\geq\lambda\int_{\mathbb{R}^N}\left((|Du_{n}|^{2}Du_n-|Du|^{2}Du)D(u_n-u)+(u_{n}^{3}-u^{3})(u_n-u)\right)\ud x\\
&\indent+C\int_{\mathbb{R}^N}(1+u_{n}^{2})|D(u_n-u)|^{2}\ud x+\lambda|u_n|_{2}^{2\alpha}\int_{\mathbb{R}^N}(u_n-u)^{2}\ud x\\
&\indent+\int_{\mathbb{R}^N}V(x)(u_{n}-u)^{2}\ud x+o_{n}(1).
\end{aligned}
\end{equation*}
Thus
$$C\int_{\mathbb{R}^N}\left(|D(u_n-u)|^{4}+(u_n-u)^{4}+|D(u_n-u)|^{2}+(u_n-u)^{2}\right)\ud x\leq o_{n}(1),$$
which implies that $\|u_n-u\|\to0$, that is, $u_n\to u$ in $E$.
\end{proof}
\begin{Lemma}\label{lem3.2}
For $\lambda\in(0,1]$, $I_{\lambda}$ has a mountain pass geometry, that is, there exist constants $R,\rho>0$(independent of $\lambda$) and $\phi\in E\setminus S_{R}$ such that $I_{\lambda}(\phi)<0$ and
$$I_{\lambda}(u)\geq\rho,\ \mbox{when}\ u\in\partial S_{R},$$
where
$$S_{R}=\left\{u\in E:\ \int_{\mathbb{R}^N}(|Du|^{2}+u^{2})\ud x+2\int_{\mathbb{R}^N}u^{2}|Du|^{2}\ud x\leq R^{2}\right\}.$$
\end{Lemma}
\begin{proof}
When $u\in\partial S_{R}$, by applying interpolation inequality, we have
$$|u|_{r}^{r}\leq|u|_{2}^{r\theta}|u|_{\frac{4N}{N-2}}^{r(1-\theta)}\leq CR^{r\theta}R^{\frac{r(1-\theta)}{2}}=CR^{\frac{r(1+\theta)}{2}},$$
where $q\in\left(2,\frac{4N}{N-2}\right)$, $r\in\left(\max\{4,q\},\frac{4N}{N-2}\right)$ and
$$
\theta=\left(\frac1r-\frac{N-2}{4N}\right)\frac{4N}{N+2}\in(0,1),\,\,\frac{r(1+\theta)}{2}>2.
$$
Using the above estimates, we have
\begin{equation*}
\begin{aligned}
I_{\lambda}(u)&\geq\frac{1}{2}\int_{\mathbb{R}^N}\sum\limits_{i,j=1}^{N}a_{ij}(u)D_{i}uD_{j}u\ud x+\frac{1}{2}\int_{\mathbb{R}^N}V(x)u^{2}\ud x\\
&\indent-\int_{\mathbb{R}^N}F(u)\ud x-\frac{1}{r}\int_{\mathbb{R}^N}|u|^{r}\ud x\\
&\geq\frac{1}{2}\int_{\mathbb{R}^N}\sum\limits_{i,j=1}^{N}a_{ij}(u)D_{i}uD_{j}u\ud x+\frac{1}{2}\int_{\mathbb{R}^N}(V(x)-\delta)u^{2}\ud x-C_{\delta}\int_{\mathbb{R}^N}|u|^{r}\ud x\\
&\geq\frac{1}{2}\int_{\mathbb{R}^N}(1+2u^{2})|Du|^{2}\ud x+\frac{1}{4}\int_{\mathbb{R}^N}V_0u^{2}\ud x-C\int_{\mathbb{R}^N}|u|^{r}\ud x\\
&\geq CR^{2}-CR^{\frac{r(1+\theta)}{2}}.
\end{aligned}
\end{equation*}
Take $R$ small enough, then we get $I_{\lambda}(u)\geq CR^{2}:=\rho$, where $\rho$ is independent of $\lambda$.

On the other hand, taking $\phi\in C^{\infty}_{0}(\mathbb{R}^N)$ and for $t>0$, we consider
\begin{equation*}
\begin{aligned}
I_{\lambda}\left(t\phi(\frac{x}{t})\right)&=\frac{\lambda t^{N}}{4}\int_{\mathbb{R}^N}|D\phi|^{4}\ud x+\frac{\lambda t^{N+4}}{4}\int_{\mathbb{R}^N}\phi^{4}\ud x+\frac{t^N}{2}\int_{\mathbb{R}^N}\sum\limits_{i,j=1}^{N}a_{ij}(t\phi)D_{i}\phi D_{j}\phi \ud x\\
&\indent+\frac{t^{N+2}}{2}\int_{\mathbb{R}^N}V(tx)\phi^{2}\ud x-t^{N}\int_{\mathbb{R}^N}F(t\phi)\ud x+\frac{\lambda t^{(1+\alpha)(N+2)}}{2(1+\alpha)}\left(\int_{\mathbb{R}^N}\phi^{2}\ud x\right)^{1+\alpha}\\
&\indent-\frac{\lambda t^{N+r}}{r}\int_{\mathbb{R}^N}|\phi|^{r}\ud x\\
&\leq\frac{ t^{N}}{4}\int_{\mathbb{R}^N}|D\phi|^{4}\ud x+\frac{C_{2}t^N}{2}\int_{\mathbb{R}^N}(1+t^{2}\phi^{2})|D\phi|^{2} \ud x+\frac{t^{N+2}}{2}\int_{\mathbb{R}^N}V_{1}\phi^{2}\ud x\\
&\indent-Ct^{N+\mu}\int_{\mathbb{R}^N}|\phi|^{\mu}\ud x+\frac{ t^{(1+\alpha)(N+2)}}{2(1+\alpha)}\left(\int_{\mathbb{R}^N}\phi^{2}\ud x\right)^{1+\alpha}\\
&\indent+\frac{\lambda t^{N+4}}{4}\int_{\mathbb{R}^N}\phi^{4}\ud x-\frac{\lambda t^{N+r}}{r}\int_{\mathbb{R}^N}|\phi|^{r}\ud x+Ct^N\\
&\leq\frac{ t^{N}}{4}\int_{\mathbb{R}^N}\left(|D\phi|^{4}+2C_{2}|D\phi|^{2}\right)\ud x+\frac{t^{N+2}}{2}\int_{\mathbb{R}^N}\left(V_{1}\phi^{2}+C_{2}\phi^{2}|D\phi|^{2}\right)\ud x\\
&\indent-Ct^{N+\mu}\int_{\mathbb{R}^N}|\phi|^{\mu}\ud x+\frac{ t^{(1+\alpha)(N+2)}}{2(1+\alpha)}\left(\int_{\mathbb{R}^N}\phi^{2}\ud x\right)^{1+\alpha}+Ct^N+\lambda \psi(t)\\
&\to-\infty,\  \mbox{as}\ t\to\infty.
\end{aligned}
\end{equation*}
Here we used the fact that $\psi(t)<0$ for $t>0$ large enough, where
$$
\psi(t):=\frac{ t^{N+4}}{4}\int_{\mathbb{R}^N}\phi^{4}\ud x-\frac{ t^{N+r}}{r}\int_{\mathbb{R}^N}|\phi|^{r}\ud x.
$$
We take $t>0$ large enough such that $I_{\lambda}(t\phi(\frac{x}{t}))<0$ and $t\phi(\frac{x}{t})\in E\setminus S_R$.
\end{proof}
\begin{Lemma}\label{lem3.3}
For any $\lambda\in(0,1]$, there exists a positive critical point $u_{\lambda}$ of $I_{\lambda}$. Furthermore , there exist two positive constants $m_1$ and $m_2$ independent of $\lambda$, such that $m_1\leq I_{\lambda}(u_{\lambda})\leq m_2$.
\end{Lemma}
\begin{proof}
To find a positive critical point of $I_{\lambda}$, we only need modify $I_{\lambda}$ to $I_{\lambda}^{+}$, which is defined by
\begin{equation*}
\begin{aligned}
I_{\lambda}^{+}(u)&=\frac{\lambda}{4}\int_{\mathbb{R}^N}\left(|Du|^{4}+u^{4}\right)\ud x+\frac{1}{2}\int_{\mathbb{R}^N}\sum\limits_{i,j=1}^{N}a_{ij}(u)D_{i}uD_{j}u\ud x+\frac{1}{2}\int_{\mathbb{R}^N}V(x)u^{2}\ud x\\
&\indent-\int_{\mathbb{R}^N}F(u^{+})\ud x+\frac{\lambda}{2(1+\alpha)}\left(\int_{\mathbb{R}^N}u^{2}\ud x\right)^{1+\alpha}-\frac{\lambda}{r}\int_{\mathbb{R}^N}|u^{+}|^{r}\ud x.
\end{aligned}
\end{equation*}
In fact, Lemma \ref{lem3.1} and \ref{lem3.2} are still valid for $I_{\lambda}^{+}$. Then from the well-known mountain pass lemma and maximum principle, we have that there exists $0<u_{\lambda}\in E$ such that $I_{\lambda}(u_{\lambda})=c_{\lambda}$. By the proof of Lemma \ref{lem3.2}, we know that $\rho$ does not depend on $\lambda$. Thus we can choose $m_1:=\rho$.

Next we take $m_2:=\max\{\max\limits_{s\in[0,1]}J_{1}(\gamma_{0}(s)),\max\limits_{s\in[0,1]}J_{2}(\gamma_{0}(s))\}$, where
\begin{equation*}
\begin{aligned}
J_{1}(u)&=\frac{1}{4}\int_{\mathbb{R}^N}\left(|Du|^{4}+u^{4}\right)\ud x+\frac{1}{2}\int_{\mathbb{R}^N}\sum\limits_{i,j=1}^{N}a_{ij}(u)D_{i}uD_{j}u\ud x+\frac{1}{2}\int_{\mathbb{R}^N}V(x)u^{2}\ud x\\
&\indent-\int_{\mathbb{R}^N}F(u^{+})\ud x+\frac{1}{2(1+\alpha)}\left(\int_{\mathbb{R}^N}u^{2}\ud x\right)^{1+\alpha}-\frac{1}{r}\int_{\mathbb{R}^N}|u^{+}|^{r}\ud x,
\end{aligned}
\end{equation*}
\begin{equation*}
\begin{aligned}
J_{2}(u)&=\frac{1}{4}\int_{\mathbb{R}^N}|Du|^{4}\ud x+\frac{1}{2}\int_{\mathbb{R}^N}\sum\limits_{i,j=1}^{N}a_{ij}(u)D_{i}uD_{j}u\ud x+\frac{1}{2}\int_{\mathbb{R}^N}V(x)u^{2}\ud x\\
&\indent-\int_{\mathbb{R}^N}F(u^{+})\ud x+\frac{1}{2(1+\alpha)}\left(\int_{\mathbb{R}^N}u^{2}\ud x\right)^{1+\alpha}.
\end{aligned}
\end{equation*}
From the above definitions, we notice that
$$I_{\lambda}(u)\leq\max\{J_{1}(u),J_{2}(u)\},\  \mbox{for all}\ u\in E.$$
Using the similar arguments in Lemma \ref{lem3.2}, there exists a path $\gamma_{0}(s)=st_{0}\phi(\frac{x}{t_{0}})$ satisfying $J_{1}(t_{0}\phi(\frac{x}{t_{0}}))<0$ and $J_{2}(t_{0}\phi(\frac{x}{t_{0}}))<0$ with $t_{0}$ large enough. Hence, we have the following estimates
$$c_{\lambda}=\inf\limits_{\gamma\in\Gamma}\max\limits_{s\in[0,1]}I_{\lambda}(\gamma(s))\leq\max\limits_{s\in[0,1]}I_{\lambda}(\gamma_{0}(s))\leq\max\{\max\limits_{s\in[0,1]}J_{1}(\gamma_{0}(s)),\max\limits_{s\in[0,1]}J_{2}(\gamma_{0}(s))\}=m_2,$$
where
$$\Gamma:=\left\{\gamma\in C([0,1],E:\ \gamma(0)=0,\ \gamma(1)=t_0\phi\left(\frac{x}{t_0}\right)\right\}.$$
\end{proof}
\noindent{\bf Proof of Theorem \ref{thm1.1}}: By Lemma \ref{lem3.3}, for any $\lambda\in(0,1]$, we know that  there exists a positive critical point $u_{\lambda}$ of $I_{\lambda}$. Furthermore, there exist two positive constants $m_1$ and $m_2$ independent of $\lambda$, such that $m_1\leq I_{\lambda}(u_{\lambda})\leq m_2$.
Let $\lambda_n\to0^{+}$ ($n\to\infty$) and $u_n$ be a positive critical point of $I_{\lambda_n}$ with $I_{\lambda_n}(u_n)\leq m_2$, according to Proposition \ref{pro2.1}, there exists $u_{0}\in H_{r}^{1}(\mathbb{R}^N)\cap L^{\infty}(\mathbb{R}^N)$, such that $u_{n}\to u_0$ in $H^{1}_{r}(\mathbb{R}^N)$, $u_{n}Du_{n}\to u_0Du_0$ in $L^{2}(\mathbb{R}^N)$ and $I_{\lambda_n}(u_n)\to I(u_0)$. Moreover $u_0$ is a positive critical point of $I$, that is to say, $u_0$ is a positive solution of (\ref{eq1.1}).
\section{Proof of Theorem \ref{thm1.2}}

\indent\indent To prove Theorem \ref{thm1.2}, we apply the following Symmetric Mountain Pass Theorem due to Ambrosetti-Rabinowitz \cite{Rab1}.
\begin{Proposition}\label{pro4.1}
Let $X$ be an infinite dimensional Banach space, $X=Y\bigoplus Z$ with $\dim Y<+\infty$. If $J\in C^{1}(X,\mathbb{R})$ satisfies P-S condition, and\\
\indent (1)\ $J(0)=0$,  $J(-u)=J(u)$  for all $u\in X$;\\
\indent (2)\  there are constants $\rho,\alpha>0$ such that $J|_{\partial B_{\rho}\cap Z}\geq\alpha$;\\
\indent (3)\  for any finite dimensional subspace $W\subset X$, there is an $R=R(W)$ such that $J\leq0$ on $W\setminus B_{R(W)}$,\\
then $J$ has a sequence of critical values $c_j\to+\infty$ as $j\rightarrow\infty$.
\end{Proposition}
\begin{Lemma}\label{lem4.1}
For any $\lambda\in(0,1]$ fixed, $I_{\lambda}$ has a sequence of critical points $\{u_{\lambda,j}\}_{j=1}^{\infty}$. Moreover for fixed $j$, there exist $\beta_j>\alpha_j$ (both of them are independent of $\lambda$) such that $I_{\lambda}(u_{\lambda,j})\in[\alpha_j,\beta_j]$ for all $\lambda\in(0,1]$.
\end{Lemma}
\begin{proof}
Basically we apply the Symmetric Mountain Pass Theorem due to Ambrosetti and Rabinowitz. Let $E_n$ be a $n$-dimensional subspace of $E$. Define
$$G_n=\{H\in C(S_{R_n}\cap E_{n},E):H\ \mbox{is odd and}\ H=id\ \mbox{on}\ \partial S_{R_n}\cap E_n\},$$
where the definition of $S_{R_{n}}$ can be found in Lemma \ref{lem3.2} and $R_n>0$ is chosen such that \\
$$I_{\lambda}(u)\leq\max\{ P_{1}(u),P_{2}(u)\}<0\ \mbox{for}\ u\in E_n\setminus S_{R_n},$$
where
\begin{equation*}
\begin{aligned}
P_{1}(u)&=\frac{1}{4}\int_{\mathbb{R}^N}|Du|^{4}\ud x+\frac{1}{2}\int_{\mathbb{R}^N}\sum\limits_{i,j=1}^{N}a_{ij}(u)D_{i}uD_{j}u\ud x+\frac{1}{2}\int_{\mathbb{R}^N}V_{1}u^{2}\ud x\\
&\indent-\int_{\mathbb{R}^N}F(u)\ud x+\frac{1}{2(1+\alpha)}\left(\int_{\mathbb{R}^N}u^{2}\ud x\right)^{1+\alpha}+\frac{1}{4}\int_{\mathbb{R}^N}u^{4}\ud x-\frac{1}{r}\int_{\mathbb{R}^N}|u|^{r}\ud x.
\end{aligned}
\end{equation*}
and
\begin{equation*}
\begin{aligned}
P_{2}(u)&=\frac{1}{4}\int_{\mathbb{R}^N}|Du|^{4}\ud x+\frac{1}{2}\int_{\mathbb{R}^N}\sum\limits_{i,j=1}^{N}a_{ij}(u)D_{i}uD_{j}u\ud x+\frac{1}{2}\int_{\mathbb{R}^N}V_{1}u^{2}\ud x\\
&\indent-\int_{\mathbb{R}^N}F(u)\ud x+\frac{1}{2(1+\alpha)}\left(\int_{\mathbb{R}^N}u^{2}\ud x\right)^{1+\alpha}.
\end{aligned}
\end{equation*}
Actually, such an $R_{n}$ can be found by the fact that in the proof of Lemma \ref{lem3.2} the element $\phi\in C^{\infty}_{0}(\mathbb{R}^N)$ is arbitrary. Note that $R_{n}$ does not depends on $\lambda$, that is to say, for all $\lambda\in(0,1]$
$$I_{\lambda}(u)<0,\ \mbox{for any}\ u\in E_{n}\cap\partial S_{R_n}.$$
Moreover, we assume $R_n>\rho_n$; see the following (\ref{eq4.2}) for the definition of $\rho_n$. Set
$$\Gamma_{j}:=\{H(S_{R_n}\cap E_n\setminus Y):H\in G_n,n\geq j,Y=-Y\subset S_{R_n}\cap E_n\ \mbox{open},\ \mbox{and}\ \gamma(Y)\leq n-j\},$$
where $\gamma(\cdot)$ is the genus. Define
$$c_{j}(\lambda)=\inf\limits_{B\in\Gamma_j}\sup\limits_{u\in B}I_{\lambda}(u),j=1,2,\cdot\cdot\cdot.$$
For any fixed $j$, define
$$\beta_{j}=\max\left\{\sup\limits_{u\in S_{R_n}\cap E_n}P_{1}(u),\sup\limits_{u\in S_{R_n}\cap E_n}P_{2}(u)\right\},\ j=1,2,\cdot\cdot\cdot.$$
Since $I_{\lambda}(u)\leq \max\{P_{1}(u),P_{2}(u)\}$ and $id\in G_n$, then we have
$$c_{j}(\lambda)=\inf\limits_{B\in\Gamma_j}\sup\limits_{u\in B}I_{\lambda}(u)\leq\sup\limits_{u\in S_{R_n}\cap E_n}I_{\lambda}(u)\leq\max\left\{\sup\limits_{u\in S_{R_n}\cap E_n}P_{1}(u),\sup\limits_{u\in S_{R_n}\cap E_n}P_{2}(u)\right\}=\beta_{j}.$$
Next we estimate the lower bound for $c_{j}(\lambda)$. By the intersection lemma in \cite{Rab1}, if $\rho<R_n$ for all $n\geq j$, then
for any $B\in\Gamma_{j}$, we have $$B\cap\partial S_{\rho}\cap E^{\perp}_{j-1}\neq\emptyset.$$
Hence
$$c_{j}(\lambda)\geq\inf\limits_{u\in\partial S_{\rho}\cap E^{\perp}_{j-1}}I_{\lambda}(u).$$
By condition $(f_1)$ and $(f_2)$, for any $\delta>0$ there exists $C=C_{\delta}>0$ such that $|F(u)|+\frac{1}{r}|u|^{r}\leq\delta u^{2}+C_{\delta}|u|^{r}$. Then for $\epsilon$ small and $u\in\partial S_{\rho}\cap E^{\perp}_{j-1}$, for any $1\leq j\leq n$,
\begin{equation*}
\begin{aligned}
I_{\lambda}(u)&\geq\frac{1}{2}\int_{\mathbb{R}^N}\sum\limits_{i,j=1}^{N}a_{ij}(u)D_{i}uD_{j}u\ud x+\frac{1}{2}\int_{\mathbb{R}^N}V(x)u^{2}\ud x\\
&\indent-\int_{\mathbb{R}^N}F(u)\ud x-\frac{1}{r}\int_{\mathbb{R}^N}|u|^{r}\ud x\\
&\geq\frac{1}{2}\int_{\mathbb{R}^N}\sum\limits_{i,j=1}^{N}a_{ij}(u)D_{i}uD_{j}u\ud x+\frac{1}{2}\int_{\mathbb{R}^N}(V(x)-\delta) u^{2}\ud x\\
&\indent-C_{\delta}\int_{\mathbb{R}^N}|u|^{r}\ud x\\
&\geq\frac{1}{2}\int_{\mathbb{R}^N}(1+2u^{2})|Du|^{2}\ud x+\frac{1}{4}\int_{\mathbb{R}^N}V_0u^{2}\ud x\\
&\indent-C\int_{\mathbb{R}^N}|u|^{r}\ud x\\
&\geq C\rho^{2}-C\rho^{\frac{3}{4}r}.
\end{aligned}
\end{equation*}
Since $r>4$, then let $\rho_{n}$ small enough such that for any $u\in\partial S_{\rho}\cap E^{\perp}_{j-1}$, we have
\begin{equation}\label{eq4.2}
I_{\lambda}(u)\geq C\rho_{n}^{2}>0.
\end{equation}
Finally, we apply Proposition \ref{pro4.1} to obtain that $c_{j}(\lambda)$, $j=1,2,\cdot\cdot\cdot$, are critical values of $I_{\lambda}$ and $c_{j}(\lambda)\to+\infty$ as $j\to+\infty$.
\end{proof}

\noindent{\bf The proof of Theorem \ref{thm1.2}.}
Due to the above lemma, for any $\lambda\in(0,1]$ fixed, $I_{\lambda}$ has a sequence of critical points $\{u_{\lambda,j}\}_{j=1}^{\infty}$ with $I_{\lambda}(u_{\lambda,j})=c_{\lambda}^{j}\leq\beta_j$.  Then from Lemma \ref{lem2.1}, it implies that $\lambda\int_{\mathbb{R}^N}\left(|Du_{\lambda,j}|^{4}+u_{\lambda,j}^{4}\right)\ud x\leq C$, $\int_{\mathbb{R}^N}(1+u_{\lambda,j}^{2})|Du_{\lambda,j}|^{2}\ud x\leq C$ and $\int_{\mathbb{R}^N}u_{\lambda,j}^{2}\ud x\leq C$. As $\lambda\to0^{+}$, using Proposition \ref{pro2.1}, we know that there exists a critical point $u_{0,j}\in H_{R}^{1}(\mathbb{R}^N)\cap L^{\infty}(\mathbb{R}^N)$ of $I$ such that $u_{\lambda,j}\to u_{0,j}$ in $H_{R}^{1}(\mathbb{R}^N)$, $u_{\lambda,j}Du_{\lambda,j}\to u_{0,j}Du_{0,j}$ in $L^{2}(\mathbb{R}^N)$, $\lambda\int_{\mathbb{R}^N}(|Du_{\lambda,j}|^{p}+|u_{\lambda,j}|^{p})\ud x\to0$, and $I_{\lambda}(u_{\lambda,j})\to I(u_{0,j})$. Since $c^{j}_{\lambda}=I_{\lambda}(u_{\lambda,j})\leq\beta_j$, we can assume $c_{\lambda}^{j}\to c_{\ast}^{j}$, as $\lambda\to0^{+}$.

 Notice that $c_{\ast}^{j}=I(u_{0,j})$. Next we claim that $c_{\ast}^{j}\to+\infty$ as $j\to\infty$. Indeed, we have
\begin{equation*}
\begin{aligned}
I_{\lambda}(u)&\geq\frac{1}{2}\int_{\mathbb{R}^N}\sum\limits_{i,j=1}^{N}a_{ij}(u)D_{i}uD_{j}u\ud x+\frac{1}{2}\int_{\mathbb{R}^N}V(x)u^{2}\ud x\\
&\indent-\int_{\mathbb{R}^N}F(u)\ud x-\frac{1}{r}\int_{\mathbb{R}^N}|u|^{r}\ud x\\
&\geq\frac{1}{2}\int_{\mathbb{R}^N}\sum\limits_{i,j=1}^{N}a_{ij}(u)D_{i}uD_{j}u\ud x+\frac{1}{2}\int_{\mathbb{R}^N}(V(x)-\delta) u^{2}\ud x\\
&\indent-C_{\delta}\int_{\mathbb{R}^N}|u|^{r}\ud x\\
&\geq\frac{1}{2}\int_{\mathbb{R}^N}(1+2u^{2})|Du|^{2}\ud x+\frac{1}{4}\int_{\mathbb{R}^N}V_0u^{2}\ud x\\
&\indent-C\int_{\mathbb{R}^N}|u|^{r}\ud x:=\widetilde{I}(u).
\end{aligned}
\end{equation*}
Since $r\in(4,\frac{4N}{N-2})$, by dual method in \cite{Col1,Shen1} and similar arguments as those in Lemma \ref{lem4.1}, it is easy to prove there exists a sequence of critical values $\{d^{j}\}$ of $\widetilde{I}(u)$, which can be defined by
$$d^{j}:=\inf\limits_{B\in \Gamma_{j}}\sup\limits_{u\in B}\widetilde{I}(u),$$
where $\Gamma_j$ is defined in Lemma \ref{lem4.1}. Moreover, Symmetric Mountain Pass Theorem (Proposition \ref{pro4.1}) implies that $d^{j}\to+\infty$ as $j\to+\infty$. From the definitions of $c_{\lambda}^{j}$ and $d^{j}$, we have $c_{\lambda}^{j}\geq d^{j}$. Finally, taking $\lambda\to0^{+}$ we get $c^{j}_{\ast}\geq d^{j}\to+\infty$ as $j\to+\infty$. So the original problem has infinitely many solutions. This ends the proof.\\

\end{document}